\documentclass[12pt]{article}

\usepackage[a4paper,centering,bindingoffset=0cm,hmargin=2cm,vmargin=2cm,includeheadfoot]{geometry}
\usepackage[font=footnotesize,labelfont=bf,width=0.75\textwidth,labelsep=period]{caption}
\usepackage{amsmath}
\usepackage{amssymb,bbm,color}
\usepackage{latexsym}
\usepackage{graphicx}
\usepackage{color,soul}
\usepackage{enumitem}
\usepackage[english]{babel}
\usepackage{authblk}

\usepackage{aliascnt}
\usepackage[hyperref,amsmath,thmmarks]{ntheorem}

\newaliascnt{proposition}{lemma}
\newtheorem{proposition}[proposition]{Proposition}
\aliascntresetthe{proposition}

\newaliascnt{theorem}{lemma}
\newtheorem{theorem}[theorem]{Theorem}
\aliascntresetthe{theorem}

\newaliascnt{assumption}{lemma}

\aliascntresetthe{assumption}

\newaliascnt{remark}{lemma}

\aliascntresetthe{remark}

\newaliascnt{definition}{lemma}
\newtheorem{definition}[definition]{Definition}
\aliascntresetthe{definition}

\theoremstyle{nonumberplain}
\theoremseparator{:}
\theoremheaderfont{\normalfont\itshape}
\theorembodyfont{\normalfont}
\theoremsymbol{\ensuremath{\square}}
\newtheorem{proof}{Proof}

\usepackage[pdftex,colorlinks=true,linkcolor=blue,unicode,pdfpagelabels]{hyperref}
\let\RE\Re
\let\Re=\undefined
\DeclareMathOperator{\Re}{\RE e}
\let\IM\Im
\let\Im=\undefined
\DeclareMathOperator{\Im}{\IM m}
\newcommand{\R}{\mathbbm R}
\renewcommand{\C}{\mathbbm C}

\newcommand{\e}{\mathrm e}

\usepackage{bm}

\DeclareMathOperator{\curl}{\nabla \times}

\newcommand{\Hone}{\b H^1}
\newcommand{\Eone}{\b E^1}
\newcommand{\Htwo}{\b H^2}
\newcommand{\Etwo}{\b E^2}
\newcommand{\He}{\b H^{ext}}
\newcommand{\Ee}{\b E^{ext}}

\newcommand{\he}{h^{ext}}
\newcommand{\ee}{e^{ext}}
\newcommand{\n}{\b n}

\renewcommand{\b}[1]{\ensuremath{\mathbf{#1}}} 

\allowdisplaybreaks

\usepackage{array,multirow}
\usepackage{hhline}
{\renewcommand{\arraystretch}{1.5}
\usepackage[table]{xcolor}
\usepackage{authblk}

\usepackage{tikz}
\usetikzlibrary{positioning}

\allowdisplaybreaks

\begin{document}

\title{ The direct electromagnetic scattering problem by a piecewise constant inhomogeneous cylinder at oblique incidence }

\author{ D. Gintides$^1$\thanks{dgindi@math.ntua.gr}, \,  S. Giogiakas$^1$\thanks{giogiakas@mail.ntua.gr} \, and L. Mindrinos$^2$\thanks{leonidas.mindrinos@univie.ac.at}
}

\affil{$^1$Department of Mathematics, National Technical University of Athens, Greece \\
$^2$Faculty of Mathematics, University of Vienna, Austria}


\maketitle

 \begin{abstract}
 We consider the solvability of the direct scattering problem of an obliquely incident time-harmonic electromagnetic wave by a piecewise constant inhomogeneous, penetrable and infinitely long cylinder. We prove the existence and uniqueness of the solution using properties of the boundary value operators and the integral equation method. For the numerical solution, we apply a collocation method and we approximate the singular integral operators using quadrature rules. We show convergence of the numerical scheme for the interior and the scattered fields both in the near- and far-field regime.
 \end{abstract}

\section{Introduction}\label{sec_intro}

The scattering problem of a time-harmonic electromagnetic wave by a infinitely long cylinder, either penetrable or not, has attracted considerable interest from researchers working in different fields, see \cite{Lee16, LucPanSche10, ShaWu16, Tsa18} for some recent applications. From a mathematical point of view, the solution of the direct problem, analytical and  numerical, needs special treatment even though it is based on classical techniques. This work adds to a series of papers dealing with direct scattering problems under similar conditions \cite{GinMin16, Min18, NakWan13, WanNak12}.

The original problem is formulated in three dimensions for an inhomogeneous scatterer, parallel to one axis. The inhomogeneity of the scatterer is described by piecewise constant electric and magnetic material parameters. Here, we deal with a two-layered medium but the proposed method can be generalized to more layers.  This setup generalizes the homogeneous case and it can be seen as a first step towards dealing with problems for inhomogeneous cylinders.

Initially, the interactions between the electric and magnetic fields inside the cylinder are described by a system of Maxwell equations in every homogeneous sub-domain. Then, the symmetry and the properties of the cylinder reduces the set of equations to a system of two-dimensional Helmholz equations only for the third components of the fields. The drawback of this dimension reduction is the complexity appearing at the new transmission boundary conditions evolving now combinations of the normal and the tangential derivatives of the fields. 

To prove uniqueness we consider Green's theorem and Rellich's Lemma for the exterior fields and we formulate an equivalent interior boundary value problem that satisfies the Shapiro-Lopatinskij condition. In addition, this boundary value problem is normal and self-adjoint resulting to a boundary value operator with real and discrete spectrum \cite{Wloka}. The existence of solutions follows from the integral representation of the solution and the Riesz-Fredholm theory. The boundary integral equation method for solving transmission problems for the Helmholtz equation has been extensively investigated, see \cite{CosSte85, KleMar88} for some early works and \cite{CakKre17, Hsi11, Yin17} for some recent examples. We search the solution combining the direct (Green's formula) and the indirect (single layer ansatz) methods. Then, the fields solve the direct problem if the unknown densities are solutions of a system of boundary integral equations. The system consists of integral operators with singular kernels. 

For the numerical implementation we first handle the singularity of the integral operators using standard decompositions and quadrature rules \cite{Kre14}. Then, we collocate the system of equations at equidistant grid  points resulting to a well-conditioned linear system, reflecting the well-posedness of the direct problem. 

This paper is organized as follows. In \autoref{sec_problem}, we formulate the direct scattering problem, the governing equations, the transmission boundary conditions and the radiation conditions for the scattered fields. The well-posedness of the derived problem is investigated in \autoref{sec_unique}. We first show uniqueness and then existence using the mapping properties of the boundary value operator. We use results from \cite{Wloka} to write the boundary value problem as an initial value problem. Then using \autoref{Appendix}, we prove that  the Shapiro-Lopatinskij condition is satisfied on the boundaries. In the last section, we present different numerical results.  We check the correctness of the proposed numerical scheme using test examples with analytic solutions and then we consider examples modeling scattering problems with obliquely incident waves.

\section{Formulation of the problem}\label{sec_problem}

The cylinder $\Omega_{int}\subset\mathbb{R}^3$ is infinitely long and parallel to the $z-$axis. It is piecewise constant inhomogeneous and admits the form $\Omega_{int}:=\Omega_{int}^1 \cup \Omega_{int}^2,$ with $\Omega_{int}^1\cap\Omega_{int}^2=\emptyset.$ The exterior domain $\Omega_{ext} = \mathbb{R}^3 \setminus \overline{\Omega_{int}}$ is an unbounded homogeneous medium characterized by the electric permittivity $\epsilon_0$ and the magnetic permeability $\mu_0.$ Each layer $\Omega^j_{int}, \, j = 1,2$ is bounded and homogeneous with 
constant material properties $\epsilon_j$ and $\mu_j ,$ respectively. The boundary $\partial \Omega$ is sufficiently smooth and consists of two disjoint surfaces $\partial\Omega_0$ and $\partial\Omega_1,$ with $\partial\Omega_1$ being in the interior of $\partial\Omega_0.$

Then, the exterior electric and magnetic fields $\Ee , \He : \Omega_{ext} \rightarrow \C^3,$  and the interior fields $\Eone , \Hone : \Omega^1_{int} \rightarrow \C^3,$ and $\Etwo , \Htwo : \Omega^2_{int} \rightarrow \C^3,$ satisfy  the system of Maxwell equations 
\begin{equation}\label{eq_Maxwell}
  \begin{aligned}
\curl \Ee - i \omega \mu_0  \He  &= 0, & \curl \He + i \omega \epsilon_0 \Ee &= 0,  & &\mbox{in  } \Omega_{ext} ,\\
\curl \Eone - i \omega \mu_1  \Hone &= 0, & \curl \Hone + i \omega \epsilon_1  \Eone  &= 0,  
&  &\mbox{in  }  \Omega^1_{int} , \\
\curl \Etwo - i \omega \mu_2  \Htwo &= 0, & \curl \Htwo + i \omega \epsilon_2  \Etwo  &= 0,  
& &\mbox{in  }  \Omega^2_{int} ,
\end{aligned}
\end{equation}
where $\omega >0$ is the frequency. At the boundary $\partial \Omega$ we impose transmission conditions of the form
\begin{equation}\label{bound_cond}
  \begin{aligned}
\hat\n \times \Eone  &=   \hat\n \times \Ee,  & \hat\n \times \Hone  &=   \hat\n \times \He,  & \mbox{on  } \partial \Omega_0 , \\
\hat\n \times \Eone  &=   \hat\n \times \Etwo, & \hat\n \times \Hone &=   \hat\n \times \Htwo, & \mbox{on  } \partial \Omega_1,
\end{aligned}
\end{equation}
where $\hat\n$ is the unit normal vector.

The scatterer is illuminated by a time-harmonic obliquely incident wave, meaning a transverse magnetic polarized electromagnetic plane wave. The cylindrical symmetry of the medium and its specific structure allows for reduction of the 3D scattering problem \eqref{eq_Maxwell} -- \eqref{bound_cond} to a 2D problem only for the $z-$components of the fields, see for instance \cite{GinMin16, NakWan13, WanNak12}.

We define by $\theta\in (0,\pi)$ the incident angle with respect to the negative $z-$axis and $\phi\in[0,2\pi]$ the polar angle of the incident direction. We set $k_0=\omega\sqrt{\mu_0\epsilon_0}$ the wave number of $\Omega_{ext}$ and we define by $k_j=\omega\sqrt{\mu_j\epsilon_j},$ the wave number of $\Omega_{int}^j, \, j=1,2.$ We assume 
 $k_1\neq k_2$. Let $\beta=k_0\cos\theta,$ then we define $ \kappa_j^2=k_j^2-\beta^2,\, j=0,1,2$. The coefficients are chosen such that $\kappa_1^2,\kappa_2^2>0$.  We denote by $\Omega_0$ the horizontal cross section of the exterior domain and by  $\Omega$ the horizontal cross section of the cylinder, given by $\Omega = \Omega_1 \cup \Omega_2.$ The simply connected domain $\Omega_2$ has a $C^2$ closed boundary $\Gamma_1$ and the doubly connected domain $\Omega_1$ admits the outer smooth boundary $\Gamma_0.$

We set $\b x =(x,y)\in\mathbb{R}^2$. We define by  $e^{ext}( \b x ), \, h^{ext}(\b x ),\,\b x\in\Omega_0$ the $z-$components of the
exterior electric and magnetic fields, respectively. The fields $e^j (\b x)$ and $ h^j (\b x )$ describe the  $z-$components of the electric and magnetic  interior fields, for $\b x\in\Omega_j , \, j =1,2,$ respectively. Following \cite{WanNak12}, we know that the fields satisfy the system of Helmholtz equations
\begin{equation}\label{helm}
\begin{aligned}
\Delta \ee + \kappa^2_0 \,\ee &= 0,  &
\Delta \he + \kappa^2_0 \,\he &= 0,  & \mbox{in  } \Omega_0 , \\
\Delta e^1 + \kappa^2_1 \,e^1 &= 0, &
\Delta h^1 + \kappa^2_1 \,h^1 &= 0, & \mbox{in  } \Omega_1 ,\\
\Delta e^2 + \kappa^2_2 \,e^2 &= 0, &
\Delta h^2 + \kappa^2_2 \,h^2 &= 0, & \mbox{in  } \Omega_2 .
\end{aligned}
\end{equation}

The boundary conditions \eqref{bound_cond} can also be rewritten in the two-dimensional setting \cite{GinMin16, WanNak12}. Let $\b n = (n_1 , n_2)$ and $\bm \tau = (-n_2 , n_1)$ be the normal and tangent vector on $\Gamma _j , \, j=0,1,$ respectively. The vector $\b n$ on $\Gamma_j$ points into $\Omega_j , \, j=0,1.$
We define $\tfrac{\partial}{\partial n } = \b n \cdot \nabla, \, \tfrac{\partial}{\partial \tau } = \bm \tau \cdot \nabla ,$ 
where $\nabla$ is the two-dimensional gradient and we set
 $$\tilde\mu_j = \frac{ \mu_j }{ \kappa_j^2} , \quad \tilde\epsilon_j = \frac{\epsilon_j }{ \kappa_j^2} , \quad \beta_j = \frac{\beta }{ \kappa_j^2 }, \quad \mbox{for}\quad  j=0,1,2.$$

The transmission conditions \eqref{bound_cond} take the form
\begin{equation}\label{boundary_ext}
\begin{aligned}
e^1 &= \ee , & \mbox{on   } \Gamma_0 , \\
\tilde\mu_1 \omega \frac{\partial h^1}{\partial n }  + \beta_1 \frac{\partial e^1}{\partial \tau } &= \tilde\mu_0 \omega \frac{\partial \he}{\partial n }  + \beta_0 \frac{\partial \ee}{\partial \tau }, & \mbox{on   } \Gamma_0 , \\
h^1 &= \he , & \mbox{on   } \Gamma_0 , \\
\tilde\epsilon_1 \omega \frac{\partial e^1}{\partial n }  - \beta_1 \frac{\partial h^1}{\partial \tau } &= \tilde\epsilon_0 \omega \frac{\partial \ee}{\partial n }  - \beta_0 \frac{\partial \he}{\partial \tau }, & \mbox{on   } \Gamma_0 , \\
e^1 &= e^2 , & \mbox{on   } \Gamma_1 , \\
\tilde\mu_1 \omega \frac{\partial h^1}{\partial n }  + \beta_1 \frac{\partial e^1}{\partial \tau } &= \tilde\mu_2 \omega \frac{\partial h^2}{\partial n }  + \beta_2 \frac{\partial e^2}{\partial \tau }, & \mbox{on   } \Gamma_1 , \\
h^1 &= h^2 , & \mbox{on   } \Gamma_1 , \\
\tilde\epsilon_1 \omega \frac{\partial e^1}{\partial n }  - \beta_1 \frac{\partial h^1}{\partial \tau } &= \tilde\epsilon_2 \omega \frac{\partial e^2}{\partial n }  - \beta_2 \frac{\partial h^2}{\partial \tau }, & \mbox{on   } \Gamma_1 .
\end{aligned}
\end{equation}

The exterior fields consist of the incident fields $e^{inc}, \, h^{inc}$ and the scattered fields $e^0, h^0,$ meaning $e^{ext}=e^{inc}+e^0$ and $h^{ext}=h^{inc}+h^0$. The incident wave reduces to the fields
\begin{equation}\label{incident}
e^{inc} (\b x) = \frac1{\sqrt{\epsilon_0}} \sin \theta \, \e^{i\kappa_0 \,\b x \cdot (\cos \phi, \,\sin \phi )}, \quad
h^{inc} (\b x) = 0.
\end{equation}
The scattered fields must satisfy the Sommerfeld radiation condition
\begin{equation}\label{radiation}
\begin{aligned}
\lim_{r \rightarrow \infty} \sqrt{r} \left( \frac{\partial e^0}{\partial r} - i\kappa_0 e^0 \right) =0 , \quad
 \lim_{r \rightarrow \infty} \sqrt{r} \left( \frac{\partial h^0}{\partial r} - i\kappa_0 h^0 \right) =0 , \\
\end{aligned}
\end{equation}   
where $r = |\b x |,$ uniformly over all directions.  This radiation condition results to the following asymptotic behavior \cite{ColKre13}
\begin{equation*}
e^0  (\b x ) = \frac{\e^{i\kappa_0 r }}{\sqrt{r}} e^\infty (\b{\hat{x}}) + \mathcal{O} ( r^{-3/2}) , \quad h^0  (\b x ) = \frac{\e^{i\kappa_0 r }}{\sqrt{r}} h^\infty (\b{\hat{x}}) + \mathcal{O} (r^{-3/2}),
\end{equation*}
where $\b{\hat{x}} = \b x / r \in S,$ with $S$ being the unit circle. The pair $(e^\infty, h^\infty)$ is called the far-field pattern of the scattered wave.

\section{Well-posedness of the problem}\label{sec_unique}
In this section we study the well-posedness of the direct scattering problem \eqref{helm} -- \eqref{radiation}. We first address the problem of unique solvability. To do so, we prove the Shapiro-Lopatinskij condition for the following interior transmission problem. For this, we have to exclude a certain discrete set of wavenumbers $\kappa_1$ and $\kappa_2$ in $\Omega_{1,2}:=\Omega_1\cup\overline{\Omega}_2$. 

We set the piecewise constant density
\[
p( \b x)=
\begin{cases}
1,  & \b x \in \Omega_1,\\
\kappa_2^2/ \kappa_1^2 ,  & \b x \in \Omega_2,
\end{cases}
\]
and we define the operator $\b A = - ( \Delta+\lambda p ) \b I_2,$ for $\lambda\in \R,$ where $\b I_2$ is the $2\times 2$ identity matrix. We consider the following Dirichlet eigenvalue problem 
\begin{equation}\label{bounvalprob}
\begin{aligned}
 \b A \b u^j &=0, &\mbox{in }  \Omega_j ,\\
\b M \b u^1 &=0, &\mbox{on }\Gamma_0 , \\ 
\b M^1 \b u^1- \b M^2 \b u^2 &=0, &\mbox{on }  \Gamma_1 ,
\end{aligned}
\end{equation}
where $\b u^j = (e^j ,\, h^j )^\top, \, j=1,2.$ The boundary operators are given by
   \begin{equation*} 
   \b M =\b I_2, \quad \mbox{on }\Gamma_0,\quad \mbox{and} \quad
\b M^j = \begin{pmatrix}
 1 & 0 \\
 \beta_j \frac{\partial}{\partial\tau} & \tilde{\mu}_j \omega \frac{\partial}{\partial n}\\ 
 0 & 1  \\
 \tilde{\epsilon}_j \omega \frac{\partial}{\partial n} & -\beta_j \frac{\partial}{\partial\tau}
   \end{pmatrix}, \quad \mbox{on }\Gamma_1.
\end{equation*}

The interior eigenvalue problem \eqref{bounvalprob} is a system with a Dirichlet condition on the exterior boundary $\Gamma_0$ and a special transmission condition on the interior boundary $\Gamma_1$. To prove ellipticity and then discreteness of the eigenvalues for this interior problem it is enough to show that the operator $ \b A$ is elliptic, properly elliptic and the Shapiro-Lopatinskij condition is satisfied on both boundaries. First, we prove that the operator $\b A$ is elliptic and properly elliptic.

We observe that the principal symbol of the operator $\b A$ is given by
\begin{equation*}
\b A_0 = \begin{pmatrix}
 \xi_1^2+\xi_2^2 & 0 \\
   0 & \xi_1^2+\xi_2^2
   \end{pmatrix}. 
   \end{equation*}
   The operator $\b A$ is elliptic since $\det \b A_0=(\xi_1^2+\xi_2^2)^2\neq 0,$ for all $\bm \xi=(\xi_1,\xi_2)\in\mathbb{R}^2 \setminus \{0\}$  \cite[Definition 9.21]{WloRowLaw}. In addition, the operator $\b A$ is properly elliptic since the determinant of $\b A_0$ has the root $\xi_2=i|\xi_1|$ of multiplicity two in the upper half plane $\Im \xi_2>0,$ and the root $\xi_2=-i|\xi_1|$ of multiplicity two in the lower half plane $\Im \xi_2<0$ \cite[Definition 9.24]{WloRowLaw}. Next, we show that the Shapiro-Lopatinskij condition is satisfied on the boundaries $\Gamma_0$ and $\Gamma_1$, so that the boundary value problem \eqref{bounvalprob} is elliptic. The necessary results are summarized in  \autoref{Appendix}.
 
\begin{proposition}\label{prop_ell}
If $\mu_1 \neq \mu_2,$ then the eigenvalue problem \eqref{bounvalprob} is elliptic and the set of eigenvalues is discrete. 
\end{proposition}

\begin{proof}
We have seen that the operator $\b A$ is elliptic and properly elliptic for all $\b x\in\Omega_1\cup\Omega_2$. Then, it is enough to show that the Shapiro-Lopatinskij condition is satisfied on $\Gamma_0$ and $\Gamma_1$ \cite{WloRowLaw}.

Let $\b M_0 = \b I_2$ be the principal symbol of $\b M$ and $\b A_{co} = \b A_0$ the cofactor matrix of $\b A_0.$ From \autoref{equiv}, we see that the Shapiro-Lopatinskij condition is satisfied on $\Gamma_0$ since the rows of the matrix $\b M_0\b A_{co}$ are linearly independent. 

To show the condition on $\Gamma_1,$ we choose a coordinate system such that $\b x_1=(0,0) \in\Gamma_1$ with 
 unit normal and tangent vectors $\b n_1=(0,-1)$ and $\bm \tau_1=(1,0),$ at $\b x_1,$ respectively. In this coordinate system, the principal symbols of the operators are given by \cite{WloRowLaw}
\begin{equation*}
\b M^j_0 = \begin{pmatrix}
 1 & 0 \\
 i\beta_j \xi_1 & i\tilde{\mu}_j\omega\xi_2 \\
 0 & 1 \\ 
 i\tilde{\epsilon}_j \omega\xi_2 & -i\beta_j \xi_1 
   \end{pmatrix}.  
\end{equation*}

To verify the Shapiro-Lopatinskij condition on $\Gamma_1$, we have to show that 
\begin{equation*}
\b M_0^1 \b A_{co} \equiv 0(\mbox{mod}\, a^-), \quad \mbox{and} \quad
\b M_0^2 \b A_{co} \equiv 0(\mbox{mod}\, a^+), 
\end{equation*} 
where $\alpha^{+}(\xi_1,\xi_2)=(\xi_2- i|\xi_1|)^2$ and $\alpha^{-}(\xi_1,\xi_2)=(\xi_2+ i|\xi_1|)^2$ \cite{RoitShef, Sheftel}.

Following \cite{NakWan13}, we assume that there exist four constants $a_1,\,a_2,\, a_3$ and $a_4$ such that 
 \begin{subequations}
 \begin{alignat}{2}
(\xi_1^2+\xi_2^2)\left[a_1\begin{pmatrix}
   1\\ 
   0
   \end{pmatrix}+a_2\begin{pmatrix}
   -i\beta_1\xi_1 \\ 
   i\tilde{\mu}_1\omega\xi_2 
   \end{pmatrix}+a_3\begin{pmatrix}
   0\\ 
   1
    \end{pmatrix}+a_4\begin{pmatrix}
    i\tilde{\epsilon}_1\omega\xi_2 \\
     i\beta_1\xi_1 
       \end{pmatrix}
    \right] & \equiv 0(\mbox{mod}\,a^{-}), \label{eq_system1}\\
    (\xi_1^2+\xi_2^2)\left[a_1\begin{pmatrix}
   1\\ 
   0
   \end{pmatrix}+a_2\begin{pmatrix}
   -i\beta_2\xi_1 \\ 
   i\tilde{\mu}_2\omega\xi_2 
   \end{pmatrix}+a_3\begin{pmatrix}
   0\\ 
   1
    \end{pmatrix}+a_4\begin{pmatrix}
    i\tilde{\epsilon}_2\omega\xi_2 \\
     i\beta_2\xi_1 
       \end{pmatrix}
    \right] & \equiv 0(\mbox{mod}\,a^{+}). \label{eq_system2}
\end{alignat} 
\end{subequations}
where each line of the equations \eqref{eq_system1} -- \eqref{eq_system2} have remainder zero, when they are divided by $a^-$ and $a^+$ , respectively. From \eqref{eq_system1}, we see that there exist  two polynomials $p(\xi_2)=p_1\xi_2+p_2$, and $q(\xi_2)=q_1\xi_2+q_2$ satisfying 
\begin{multline}\label{eq_coef}
(\xi_2+i|\xi_1|)(\xi_2-i|\xi_1|)\left[a_1\begin{pmatrix}
   1\\ 
   0
   \end{pmatrix}+a_2\begin{pmatrix}
   -i\beta_1\xi_1 \\ 
   i\tilde{\mu}_1\omega\xi_2 
   \end{pmatrix}+a_3\begin{pmatrix}
   0\\ 
   1
    \end{pmatrix}+a_4\begin{pmatrix}
    i\tilde{\epsilon}_1\omega\xi_2 \\
     i\beta_1\xi_1 
       \end{pmatrix}
    \right] \\
    =   \begin{pmatrix}
   p_1\xi_2+p_2\\ 
   q_1\xi_2+q_2
   \end{pmatrix}(\xi_2+i|\xi_1|)^2.
    \end{multline}
    
Comparing the coefficients in both sides of the first equation, we have
\[
p_1=ia_4\tilde{\epsilon}_1\omega,\quad p_2=-a_1+ia_2\beta_1\xi_1, \quad \mbox{and} \quad
p_2+ip_1|\xi_1|=a_1-ia_2\beta_1\xi_1+a_4\tilde{\epsilon}_1\omega|\xi_1|,
\]
resulting to 
\begin{equation*}
a_1-ia_2\beta_1\xi_1+a_4\tilde{\epsilon}_1\omega|\xi_1|=0.
\end{equation*}
Similarly, from the second equation in \eqref{eq_coef} we have
\begin{equation*}
a_2\tilde{\mu}_1\omega|\xi_1|+a_3+ia_4\beta_1\xi_1=0. 
\end{equation*}
Following the same steps, from the system \eqref{eq_system2}, we get
\begin{align*}
a_1-ia_2\beta_2\xi_1-a_4\tilde{\epsilon}_2\omega|\xi_1| &=0,\\
a_2\tilde{\mu}_2\omega|\xi_1|-a_3-ia_4\beta_2\xi_1 &=0. 
\end{align*}
The last four equations result to the system
\[
\begin{pmatrix}
1 & -i\beta_1\xi_1& 0 & \tilde{\epsilon_1}\omega|\xi_1|\vspace{0.1cm}\\
0 & \tilde{\mu}_1\omega|\xi_1| & 1 & i\beta_1\xi_1\vspace{0.1cm}\\
1 & -i\beta_2\xi_1 & 0 & -\tilde{\epsilon}_2\omega|\xi_1|\vspace{0.1cm}\\
0 & \tilde{\mu}_2\omega|\xi_1| & -1 & -i\beta_2\xi_1
\end{pmatrix} \begin{pmatrix}
a_1 \\ a_2 \\ a_3 \\ a_4
\end{pmatrix} = \begin{pmatrix}
0 \\ 0 \\ 0 \\ 0
\end{pmatrix},
\]
with determinant
\begin{equation*}
\begin{aligned}
 \mbox{Det} &=(\tilde{\mu}_2\tilde{\epsilon}_2\omega^2-\beta_2^2+\beta_1^2-\tilde{\mu}_1\tilde{\epsilon}_1\omega^2+\tilde{\mu}_1\tilde{\epsilon}_2\omega^2-\tilde{\mu}_2\tilde{\epsilon}_1\omega^2)|\xi_1|^2\\ &=\frac{\omega^2|\xi_1|^2}{\kappa_1^2\kappa_2^2}(\mu_1-\mu_2)(\epsilon_1+\epsilon_2). 
\end{aligned}
\end{equation*}

Since $\mu_1\neq \mu_2,$ the determinant is not identically zero implying that $a_1=a_2=a_3=a_4=0.$ Thus, the Shapiro-Lopatinskij condition is satisfied on $\Gamma_1.$ We define $\mathbbm L = (\b A, \b M, \b M^1, \b M^2)$. We note that the formally adjoint operator $\b A^*:(H^2(\Omega_j))^2\rightarrow (L_2(\Omega_j))^2,\,j=1,2$ is the same as the operator $\b A:(H^2(\Omega_j))^2\rightarrow (L_2(\Omega_j))^2,\,j=1,2$ \cite[Definition 10.1]{Wloka}. In addition, we observe that the boundary operators $\b M, \, \b M^j$ are real. Then, the interior boundary transmision problem which is described by the operator $\mathbbm L$  is self-adjoint  \cite[Definition 14.6]{Wloka} and the index of the operator  $\mathbbm L$ is zero \cite[Corollary 15.8]{Wloka}. Therefore, the boundary value problem \eqref{bounvalprob} is elliptic, the operator $\mathbb{L}$ has the Fredholm property and its spectrum exists, is real and discrete \cite{Agran79, Ray}. 
\end{proof} 

Next, we prove the uniqueness of the solution of the problem \eqref{helm} -- \eqref{radiation} using Green's theorem and Rellich's lemma.

\begin{theorem}\label{theorem1}
If $\mu_1 \neq \mu_2$ and $\kappa_1^2$ is not an interior Dirichlet eigenvalue for the domain $\Omega_{1,2}$, then the problem \eqref{helm} -- \eqref{radiation} admits at most one solution.
\end{theorem}

\begin{proof}
It is sufficient to show that the homogeneous version of \eqref{helm} -- \eqref{radiation} admits only the trivial solution. 
We consider a disk $S_r$ with center $(0,0)$, radius $r>0$ and boundary $\Gamma_r$, which contains $\overline{\Omega}$. We define $\Omega_r:=S_r\setminus \overline{\Omega}$. In the following, with $\bar u$ we denote the complex conjugate of the field $u.$

We apply Green's first theorem for $e^1,\overline{e^1}$ and $h^1,\overline{h^1}$ in $\Omega_1:$ 
\begin{subequations}
  \begin{alignat}{2}
  \int_{\Gamma_0\cup\Gamma_1}e^1\frac{\partial\overline{e^1}}{\partial n}ds &=\int_{\Omega_1}(|\nabla e^1|^2-\kappa_1^2|e^1|^2)dx,  \label{Greenfirst1}\\
  \int_{\Gamma_0\cup\Gamma_1}h^1\frac{\partial\overline{h^1}}{\partial n}ds &=\int_{\Omega_1}(|\nabla h^1|^2-\kappa_1^2 |h^1|^2)dx, \label{Greenfirst2}
  \end{alignat}
\end{subequations}
and in $\Omega_2$ for the fields  $e^2,\overline{e^2}$ and $h^2,\overline{h^2}:$
\begin{subequations}
\begin{alignat}{2}
\int_{\Gamma_1}e^2\frac{\partial\overline{e^2}}{\partial n}ds &=\int_{\Omega_2}(|\nabla e^2|^2-\kappa_2^2|e^2|^2) dx, \label{Greenomega21}\\
\int_{\Gamma_1}h^2\frac{\partial\overline{h^2}}{\partial n}ds &=\int_{\Omega_2}(|\nabla h^2|^2-\kappa_2^2|h^2|^2) dx.\label{Greenomega22}
\end{alignat}
\end{subequations}
The homogeneous boundary conditions take the form
\begin{subequations}
  \begin{alignat}{2}
  e^1 &= e^0, &&\mbox{ on  } \Gamma_0 ,\label{homogenous1} \\
\tilde\mu_1 \omega \frac{\partial h^1}{\partial n }  + \beta_1 \frac{\partial e^1}{\partial \tau } &= \tilde\mu_0 \omega \frac{\partial h^0}{\partial n }  + \beta_0 \frac{\partial e^0}{\partial \tau }, &&\mbox{ on   } \Gamma_0 ,\label{homogenous2} \\
h^1 &= h^0 ,\label{homogenous3} && \mbox{ on   } \Gamma_0 , \\
\tilde\epsilon_1 \omega \frac{\partial e^1}{\partial n }  - \beta_1 \frac{\partial h^1}{\partial \tau } &= \tilde\epsilon_0 \omega \frac{\partial e^0}{\partial n }  - \beta_0 \frac{\partial h^0}{\partial \tau }, &&\mbox{ on   } \Gamma_0 ,\label{homogenous4} \\
e^1 &= e^2 ,\label{homogenous5} &&\mbox{ on   } \Gamma_1 , \\
\tilde\mu_1 \omega \frac{\partial h^1}{\partial n }  + \beta_1 \frac{\partial e^1}{\partial \tau } &= \tilde\mu_2 \omega \frac{\partial h^2}{\partial n }  + \beta_2 \frac{\partial e^2}{\partial \tau }, &&\mbox{ on   } \Gamma_1 ,\label{homogenous6} \\
h^1 &= h^2 ,\label{homogenous7} &&\mbox{ on   } \Gamma_1 , \\
\tilde\epsilon_1 \omega \frac{\partial e^1}{\partial n }  - \beta_1 \frac{\partial h^1}{\partial \tau } &= \tilde\epsilon_2 \omega \frac{\partial e^2}{\partial n }  - \beta_2 \frac{\partial h^2}{\partial \tau }, &&\mbox{ on   } \Gamma_1. \label{homogenous8}
\end{alignat}
\end{subequations}
Using \eqref{Greenfirst1} and \eqref{homogenous8} we get
\begin{equation}\label{Green1e}
\begin{aligned}
\tilde{\epsilon}_1\int_{\Gamma_0}e^1\frac{\partial\overline{e^1}}{\partial n}ds &=\tilde{\epsilon}_1\int_{\Omega_1}(|\nabla e^1|^2-\kappa_1^2 |e^1|^2)dx+\tilde{\epsilon}_1\int_{\Gamma_1}e^1\frac{\partial\overline{e^1}}{\partial n}ds \\
&=\tilde{\epsilon}_1\int_{\Omega_1}(|\nabla e^1|^2-\kappa_1^2 |e^1|^2)dx+\int_{\Gamma_1}e^1\left(\tilde{\epsilon}_2\frac{\partial\overline {e^2}}{\partial n}+\frac{\beta_1}{\omega}\frac{\partial\overline{h^1}}{\partial\tau}-\frac{\beta_2}{\omega}\frac{\partial\overline{h^2}}{\partial\tau}\right)ds.
\end{aligned}
\end{equation}
Similarly, equations \eqref{Greenfirst2} and \eqref{homogenous6} result to
\begin{equation}\label{Green1m}
\begin{aligned}
\tilde{\mu}_1\int_{\Gamma_0}h^1\frac{\partial\overline{h^1}}{\partial n}ds &=\tilde{\mu}_1\int_{\Omega_1}(|\nabla h^1|^2-\kappa_1^2 |h^1|^2)dx+\tilde{\mu}_1\int_{\Gamma_1}h^1\frac{\partial\overline{h^1}}{\partial n}ds\\
&=\tilde{\mu}_1\int_{\Omega_1}(|\nabla h^1|^2-\kappa_1^2 |h^1|^2)dx+\int_{\Gamma_1}h^1\left(\tilde{\mu}_2\frac{\partial\overline{h^2}}{\partial n}+\frac{\beta_2}{\omega}\frac{\partial\overline {e^2}}{\partial\tau}-\frac{\beta_1}{\omega}\frac{\partial\overline{e^1}}{\partial\tau}\right)ds.
\end{aligned}
\end{equation}
Applying again Green's first identity and using the boundary condition \eqref{homogenous4} for the exterior fields $e^0$ and $\overline{e^0}$ in $\Omega_r$, we derive
\begin{equation}\label{Imag1}
\begin{aligned}
\tilde{\epsilon}_0\int_{\Gamma_r}e^0\frac{\partial\overline{e^0}}{\partial n}ds &=\tilde{\epsilon}_0\int_{\Omega_r}(|\nabla e^0|^2-\kappa_0^2 |e^0|^2)dx+\tilde{\epsilon}_0\int_{\Gamma_0}e^0\frac{\partial\overline{e^0}}{\partial n}ds\\
&=\tilde{\epsilon}_0\int_{\Omega_r}(|\nabla e^0|^2-\kappa_0^2 |e^0|^2)dx+\int_{\Gamma_0}e^0\left(\tilde{\epsilon}_1\frac{\partial\overline {e^1}}{\partial n}-\frac{\beta_1}{\omega}\frac{\partial\overline h^1}{\partial\tau}+\frac{\beta_0}{\omega}\frac{\partial\overline{h^0}}{\partial\tau}\right)ds.
\end{aligned}
\end{equation}
The exterior magnetic fields and \eqref{homogenous2} gives
\begin{equation}\label{Imag2}
\begin{aligned}
\tilde{\mu}_0\int_{\Gamma_r}h^0\frac{\partial\overline{h^0}}{\partial n}ds &=\tilde{\mu}_0\int_{\Omega_0}(|\nabla h^0|^2-\kappa_0^2 |h^0|^2)dx+\tilde{\mu}_0\int_{\Gamma_0}h^0\frac{\partial\overline{h^0}}{\partial n}ds\\
&=\tilde{\mu}_0\int_{\Omega_r}(|\nabla h^0|^2-\kappa_0^2 |h^0|^2)dx+\int_{\Gamma_0}h^0\left(\tilde{\mu}_1\frac{\partial\overline{h^1}}{\partial n}+\frac{\beta_1}{\omega}\frac{\partial\overline {e^1}}{\partial\tau}-\frac{\beta_0}{\omega}\frac{\partial\overline{e^0}}{\partial\tau}\right)ds.
\end{aligned}
\end{equation}
The imaginary part of \eqref{Imag1} using \eqref{Greenomega21}, \eqref{homogenous1}, \eqref{homogenous5} and \eqref{Green1e}  is given by
\begin{equation*}
\begin{aligned}
\Im\left(\tilde{\epsilon}_0\int_{\Gamma_r}e^0\frac{\partial\overline{e^0}}{\partial n}ds\right) &=\Im\left(-\frac{\beta_1}{\omega}\int_{\Gamma_0} e^1 \frac{\partial\overline {h^1}}{\partial\tau}ds+\frac{\beta_0}{\omega}\int_{\Gamma_0}e^0 \frac{\partial\overline{h^0}}{\partial\tau}ds\right)\\
&\phantom{=}+\Im\left(\frac{\beta_1}{\omega}\int_{\Gamma_1} e^1 \frac{\partial\overline {h^1}}{\partial\tau}ds-\frac{\beta_2}{\omega}\int_{\Gamma_1}e^2 \frac{\partial\overline{h^2}}{\partial\tau}ds \right).
\end{aligned}
\end{equation*}
Again, the imaginary part of \eqref{Imag2} using \eqref{Greenomega22}, \eqref{homogenous3}, \eqref{homogenous7} and \eqref{Green1m}, takes the form
\begin{equation*}
\begin{aligned}
\Im\left(\tilde{\mu}_0\int_{\Gamma_r}h^0\frac{\partial\overline{h^0}}{\partial n}ds\right) &=\Im\left(\frac{\beta_1}{\omega}\int_{\Gamma_0} h^1 \frac{\partial\overline {e^1}}{\partial\tau}ds-\frac{\beta_0}{\omega}\int_{\Gamma_0}h^0 \frac{\partial\overline{e^0}}{\partial\tau}ds\right)\\
&\phantom{=}+\Im\left(\frac{\beta_2}{\omega}\int_{\Gamma_1} h^2 \frac{\partial\overline {e^2}}{\partial\tau}ds-\frac{\beta_1}{\omega}\int_{\Gamma_1}h^1 \frac{\partial\overline{e^1}}{\partial\tau}ds \right).
\end{aligned}
\end{equation*}
Using that
\begin{equation*}
\begin{aligned}
-\int_{\Gamma_j}e^k\dfrac{\partial\overline{h^k}}{\partial\tau}ds=\overline{\int_{\Gamma_j}h^k\dfrac{\partial\overline{e^k}}{\partial\tau}ds}\quad k=0,1,2,\quad j=0,1,
\end{aligned}
\end{equation*}
we obtain from the last two equations
\begin{equation*}
\begin{aligned}
\Im\left(\tilde{\epsilon}_0\int_{\Gamma_r}e^0\dfrac{\partial\overline{e^0}}{\partial n}ds+\tilde{\mu}_0\int_{\Gamma_r}h^0\dfrac{\partial\overline{h^0}}{\partial n}ds\right)=0.
\end{aligned}
\end{equation*}
This equation,together with the radiation condition \eqref{radiation} as $r\rightarrow\infty$, and Rellich's Lemma,
results to $e^0=h^0=0,$ in $\Omega_0$ and therefore $e^0=h^0=0,$ on $\Gamma_0$
\cite{GinMin16}.

Now we have to show that also the interior fields are identical zero. The interior problem admits the form \eqref{bounvalprob}. Thus, from \autoref{prop_ell} and the assumption that $\lambda=\kappa_1^2$ is not an interior Dirichlet eigenvalue in $\Omega_{1,2}$, we obtain that $\b u^j =0,$ in $\Omega_j,$ for $j=0,1,$ resulting to $e^1 = h^1=0,$ in $\Omega_1$ and $e^2 = h^2=0,$ in $\Omega_2$. This completes the proof.
\end{proof}

For the next theorem, we need the integral representation of the solution. Thus, we present the fundamental solution of the Helmholtz equation 
\begin{equation*}
\begin{aligned}
\Phi_j({\bf{x}},{\bf{y}})=\dfrac{i}{4}H_0^{(1)}(\kappa_j|{\bf{x}}-{\bf{y}}|),\quad {\bf{x,y}}\in\Omega_j,\, {\bf{x}}\neq{\bf{y}}
\end{aligned} 
\end{equation*}
where $H_0^{(1)}$ is the Hankel function of the first kind and zero order. We introduce the single- and the double-layer potential for a continuous density $f$, given by
\begin{equation*}
\begin{aligned}
(\mathcal S_{klj} f) (\b x) &= \int_{\Gamma_j} \Phi_k (\b x,\b y) f(\b y) ds (\b y), & \b x \in\Omega_l , \\
(\mathcal D_{klj} f) (\b x) &= \int_{\Gamma_j} \frac{\partial \Phi_k}{\partial n (\b y)} (\b x,\b y) f(\b y) ds (\b y), & \b x \in\Omega_l ,
\end{aligned}
\end{equation*}
for $k,l=0,1,2$ and $j=0,1$. The single-layer potential $\mathcal S$ is continuous in $\R^2$ and its normal and tangential derivatives as $\b x \rightarrow \Gamma_j$ satisfy standard jump relations, see for instance \cite{ColKre83}. We define the operators
\begin{align*}
(S_{klj} f) (\b x) &= \int_{\Gamma_j} \Phi_k (\b x,\b y) f(\b y) ds (\b y), & \b x \in\Gamma_l ,\\
( D_{klj} f) (\b x) &= \int_{\Gamma_j} \frac{\partial \Phi_k}{\partial n (\b y)} (\b x,\b y) f(\b y) ds (\b y), & \b x \in\Gamma_l , \\
( NS_{klj} f) (\b x) &= \int_{\Gamma_j} \frac{\partial \Phi_k}{\partial n (\b x)} (\b x,\b y) f(\b y) ds (\b y), & \b x \in\Gamma_l , \\
( ND_{klj} f) (\b x) &= \int_{\Gamma_j} \frac{\partial^2 \Phi_k}{\partial n (\b x)\partial n (\b y)} (\b x,\b y) f(\b y) ds (\b y), & \b x \in\Gamma_l , \\
( TS_{klj} f) (\b x) &= \int_{\Gamma_j} \frac{\partial \Phi_k}{\partial \tau (\b x)} (\b x,\b y) f(\b y) ds (\b y), & \b x \in\Gamma_l , \\
( TD_{klj} f) (\b x) &= \int_{\Gamma_j} \frac{\partial^2 \Phi_k}{\partial \tau (\b x)\partial n (\b y)} (\b x,\b y) f(\b y) ds (\b y), & \b x \in\Gamma_l .
\end{align*}

Note that if we were using only the direct method, meaning Green's second identity, the fields would have the representations
\begin{equation*}
\begin{aligned}
u^0(\b x) &=(\mathcal D _{000}u^0)(\b x)-(\mathcal S _{000}\partial_n{u^0})(\b x), & \b x \in\Omega_0, \\
u^1(\b x) &=(\mathcal S _{110}\partial_n{u^1})(\b x)-(\mathcal D _{110}u^1)(\b x)+(\mathcal S_{111}\partial_n{u^1})(\b x)+(\mathcal D _{111}u^1)(\b x), & \b x \in\Omega_1,\\
u^2(\b x) &=(\mathcal S _{221}\partial_n{u^2})(\b x)-(\mathcal D _{221}u^2)( \b x), & \b x\in\Omega_2,
\end{aligned}
\end{equation*}
for $u=e,h$. We observe that we have 16 unknown density functions and only 8 equations, meaning the transmission
 boundary conditions \eqref{boundary_ext}. Thus, we consider a combination of the direct and the indirect methods, in order to have enough information to solve the derived system of boundary integral equations.

\begin{theorem}\label{theorem2}
If $\mu_1 \neq \mu_2$ and $\kappa_1^2$ is not an interior Dirichlet eigenvalue for $\Omega_{1,2}$, $\kappa_2^2$ is not an interior Dirichlet eigenvalue for $\Omega_2$ and $\kappa_0^2$ is not a Dirichlet eigenvalue in $\mathbb{R}^2\setminus\overline{\Omega_0}$, then the problem  \eqref{helm} -- \eqref{radiation} has a unique solution.
\end{theorem}

\begin{proof}
We combine Green's second formula for the fields  in the domains $\Omega_0$ and $\Omega_2,$ and a single layer ansatz for the interior fields in  $\Omega_1.$ We consider the forms
\begin{equation}\label{solutions}
\begin{aligned}
e^0( \b x ) &=\mathcal D_{000}\phi_0^e({\bf{x}})-\mathcal S_{000}\psi_0^e({\bf{x}}), &  {\bf{x}}\in\Omega_0,\\
h^0({\bf{x}}) &=\mathcal D_{000}\phi_0^h({\bf{x}})-\mathcal S_{000}\psi_0^h({\bf{x}}), & {\bf{x}}\in\Omega_0,\\
e^1({\bf{x}}) &=\mathcal S _{110}\psi_1^e({\bf{x}})+\mathcal S _{111}\psi_2^e({\bf{x}}), & {\bf{x}}\in\Omega_1,\\
h^1({\bf{x}}) &=\mathcal S_{110}\psi_1^h({\bf{x}})+\mathcal S_{111}\psi_2^h({\bf{x}}), & {\bf{x}}\in\Omega_1,\\
e^2({\bf{x}}) &=\mathcal S_{221}\psi_3^e({\bf{x}})-\mathcal D_{221}\phi_3^e({\bf{x}}), & {\bf{x}}\in\Omega_2,\\
h^2({\bf{x}})& =\mathcal S_{221}\psi_3^h({\bf{x}})-\mathcal D_{221}\phi_3^h({\bf{x}}), & {\bf{x}}\in\Omega_2,
\end{aligned}
\end{equation}
with $\psi_j^u:=\partial_nu^j|_{\Gamma_k} ,$ and $\phi_j^u:= u^j |_{\Gamma_k}, \, j = 0,1,2,3$, for $k=0,1,$ and  $u=e,h.$
Using the jump-relations, we see that the fields solve the direct problem if the densities satisfy the system of integral equations 
\begin{align*}
S_{100}\psi_1^e+S_{101}\psi_2^e-\left(D_{000}+\frac{1}{2}\right)\phi_0^e+S_{000}\psi_0^e &=e^{inc}
\\
\tilde{\mu}_1\omega\left(NS_{100}+\frac{1}{2}\right)\psi_1^h+\tilde{\mu}_1\omega NS_{101}\psi_2^h+{\beta_1}TS_{100}\psi_1^e+\beta_1 TS_{101}\psi_2^e-\tilde{\mu}_0\omega ND_{000}\phi_0^h\\
+\tilde{\mu}_0\omega\left(NS_{000}-\frac{1}{2}\right)\psi_0^h-\beta_0\left(TD_{000}+\frac{\partial_{\tau}}{2}\right)\phi_0^e+\beta_0 TS_{000}\psi_0^e &=\beta_0\partial_{\tau}e^{inc}
\\
S_{100}\psi_1^h+S_{101}\psi_2^h-\left(D_{000}+\frac{1}{2}\right)\phi_0^h+S_{000}\psi_0^h &=0
\\
\tilde{\epsilon}_1\omega\left(NS_{100}+\frac{1}{2}\right)\psi_1^e+\tilde{\epsilon}_1\omega NS_{101}\psi_2^e-{\beta_1}TS_{100}\psi_1^h-\beta_1 TS_{101}\psi_2^h-\tilde{\epsilon}_0\omega ND_{000}\phi_0^e\\
+\tilde{\epsilon}_0\omega\left(NS_{000}-\frac{1}{2}\right)\psi_0^e+\beta_0\left(TD_{000}+\frac{\partial_{\tau}}{2}\right)\phi_0^h-\beta_0 TS_{000}\psi_0^h &=\tilde{\epsilon}_0\omega\partial_{n}e^{inc}
\\
S_{110}\psi_1^e+S_{111}\psi_2^e-S_{211}\psi_3^e+\left(D_{211}-\dfrac{1}{2}\right)\phi_3^e &=0
\\
\tilde{\mu}_1\omega NS_{110}\psi_1^h+\tilde{\mu}_1\omega\left(NS_{111}-\frac{1}{2}\right)\psi_2^h+\beta_1 TS_{110}\psi_1^e+\beta_1 TS_{111}\psi_2^e\\
-\tilde{\mu}_2\omega\left(NS_{211}+\dfrac{1}{2}\right)\psi_3^h +\tilde{\mu}_2\omega ND_{211}\phi_3^h-\beta_2 TS_{211}\psi_3^e+\beta_2\left(TD_{211}-\frac{1}{2}\partial_{\tau}\right)\phi_3^e &=0
\\
S_{110}\psi_1^h+S_{111}\psi_2^h-S_{211}\psi_3^h+\left(D_{211}-\frac{1}{2}\right)\phi_3^h &=0
\\
\tilde{\epsilon}_1\omega NS_{110}\psi_1^e+\tilde{\epsilon}_1\omega\left(NS_{111}-\frac{1}{2}\right)\psi_2^e-\beta_1 TS_{110}\psi_1^h-\beta_1 TS_{111}\psi_2^h\\
-\tilde{\epsilon}_2\omega\left(NS_{211}+\frac{1}{2}\right)\psi_3^e+\tilde{\epsilon}_2\omega ND_{211}\phi_3^e+\beta_2 TS_{211}\psi_3^h-\beta_2\left(TD_{211}-\frac{1}{2}\partial_{\tau}\right)\phi_3^h &=0.
\end{align*}
Still we have an underdetermined system of equations. Thus, we impose the relations
\begin{equation}\label{identity}
\tilde{\epsilon}_1\psi_1^e =-\tilde{\epsilon}_0\psi_0^e,\quad \tilde{\mu}_1\psi_1^h=-\tilde{\mu}_0\psi_0^h,\quad \tilde{\mu}_1\psi_2^h=\tilde{\mu}_2\psi_3^h,\quad \tilde{\epsilon}_1\psi_2^e=\tilde{\epsilon}_2\psi_3^e.
\end{equation} 
Then, the system takes the form
\begin{equation}\label{system1}
\left( \b B+ \b C\right)\bm \phi= \b f,
\end{equation}
with
$\bm\phi=\left(\phi_0^e,\psi_1^h,\phi_0^h,\psi_1^e,\phi_3^e,\psi_2^h,\phi_3^h,\psi_2^e\right)^{\top}\in\mathbb{C}^8,$  $\b  f=\left(e^{inc},\beta_0\partial_{\tau}e^{inc},0,\tilde{\epsilon}_0\omega\partial_{n}e^{inc},0,0,0,0\right)^{\top}\in\mathbb{C}^8,$
and
\begin{equation*}
\renewcommand*{\arraystretch}{1}
\b B = \begin{pmatrix}
  -\frac{1}{2} & 0 & 0 & 0 & 0 & 0 & 0 & 0 \\ \\
   -\frac{\beta_0}{2}\partial_{\tau} & \tilde{\mu}_1\omega & 0 & 0 & 0 & 0 & 0 & 0 \\ \\
   0 & 0 & -\frac{1}{2} & 0 & 0 & 0 & 0 & 0\\ \\
   0 & 0 &\frac{\beta_0}{2}\partial_{\tau} & \tilde{\epsilon}_1\omega & 0 & 0 & 0 & 0\\ \\
   0 & 0 & 0 & 0 & -\frac{1}{2} & 0 & 0 & 0\\ \\
   0 & 0 & 0 & 0 & -\frac{\beta_2}{2}\partial_{\tau} & -\tilde{\mu}_1\omega & 0 & 0\\ \\
   0 & 0 & 0 & 0 & 0 & 0 & -\frac{1}{2} & 0\\ \\
   0 & 0 & 0 & 0 & 0 & 0 &\frac{\beta_2}{2}\partial_{\tau} & -\tilde{\epsilon}_1\omega
\end{pmatrix}.
\end{equation*}
The matrix-valued operator $\b C = (C_{kj})_{1\leq k,j \leq 8}$ has entries
\begin{align*}
C_{11} &=-D_{000}, & C_{14} &=S_{100}-\dfrac{\tilde{\epsilon}_1}{\tilde{\epsilon}_0}S_{000}, & C_{18 } &=S_{101}, \\
C_{21} &=-\beta_0 TD_{000}, & C_{22} &=\tilde{\mu}_1\omega\left(NS_{100}-NS_{000}\right), &  C_{23} &=-\tilde{\mu}_0\omega ND_{000},\\
C_{24} &=\beta_1 TS_{100}-\frac{\tilde{\epsilon}_1}{\tilde{\epsilon}_0}\beta_0 TS_{000}, &  C_{26} &=\tilde{\mu}_1\omega NS_{101}, & C_{28} &=\beta_1 TS_{101}, \\
C_{32} &=S_{100}-\dfrac{\tilde{\mu}_1}{\tilde{\mu}_0}S_{000}, &  C_{33}&=-D_{000}, &  C_{36} &=S_{101}, \\
C_{41} &=-\tilde{\epsilon}_0\omega ND_{000}, & C_{42} &=-\beta_1 TS_{100}+\beta_0\frac{\tilde{\mu}_1}{\tilde{\mu}_0}TS_{000}, & C_{43} &=\beta_0TD_{000},\\
C_{44} &=\tilde{\epsilon}_1\omega(NS_{100}-NS_{000}), & C_{46}&=-\beta_1TS_{101}, & C_{48} &=\tilde{\epsilon}_1\omega NS_{101},\\
C_{54} &=S_{110}, & C_{55}&=D_{211}, & C_{58} &=S_{111}-\frac{\tilde{\epsilon}_1}{\tilde{\epsilon}_2}S_{211},\\
C_{62} &=\tilde{\mu}_1\omega NS_{110}, & C_{64} &=\beta_1 TS_{110}, & C_{65} &=\beta_2 TD_{211},\\
C_{66} &=\tilde{\mu}_1\omega(NS_{111}-NS_{211}), & C_{67}&=\tilde{\mu}_2\omega ND_{211}, & C_{68} &=\beta_1 TS_{111}-\beta_2\frac{\tilde{\epsilon}_1}{\tilde{\epsilon}_2}TS_{211},\\
C_{72} &=S_{110}, & C_{76}&=S_{111}-\frac{\tilde{\mu}_1}{\tilde{\mu}_2}S_{211}, & C_{77} &=D_{211},\\
C_{82} &=-\beta_1 TS_{110}, & C_{84}&=\tilde{\epsilon}_1\omega NS_{110}, & C_{85}&=\tilde{\epsilon}_2\omega ND_{211},\\
C_{86} &=-\beta_1 TS_{111}+\beta_2\frac{\tilde{\mu}_1}{\tilde{\mu}_2}TS_{211}, & C_{87}&=-\beta_2TD_{211}, & C_{88} &=\tilde{\epsilon}_1\omega(NS_{111}-NS_{211}),
\end{align*}
and the rest are zero. 

The special form of $\b B$ and the boundness of the  tangential operator $\partial_{\tau}:H^{1/2}\left(\Gamma_j\right)\rightarrow H^{-1/2}\left(\Gamma_j\right)$, for $j=0,1$, results to a bounded inverse matrix $\b B^{-1}.$ Then, the system \eqref{system1}  can be written in the form
\begin{equation}\label{eq_system}
(\b I_8 + \b K) \bm\phi= \b g,
\end{equation}
where $\b I_8$ is the $8\times 8$ identity operator, $ \b K= \b B^{-1} \b C$ and 
\begin{equation*}
\b g= \b B^{-1} \b f =  \left(-2e^{inc},0,0,\dfrac{\tilde{\epsilon}_0}{\tilde{\epsilon}_1}\partial_ne^{inc},0,0,0,0\right)^{\top}.
\end{equation*}

We define the product spaces
\begin{equation*}
\begin{aligned}
H_1 &:=\left(H^{1/2}(\Gamma_0) \times H^{-1/2}(\Gamma_0)\right)^2\times\left(H^{1/2}(\Gamma_1)\times H^{-1/2}(\Gamma_1)\right)^2 ,\\
H_2 &:=\left(H^{-1/2}(\Gamma_0) \times H^{-3/2}(\Gamma_0)\right)^2\times\left(H^{-1/2}(\Gamma_1)\times H^{-3/2}(\Gamma_1)\right)^2,
\end{aligned}
\end{equation*}
and using the mapping properties of the integral operators \cite{ColKre13, Kre14} we see that the operator $\b K:H_1\rightarrow H_2$ is compact. Now we show that this operator is also injective. 

Let $\bm\phi$ solve $(\b I_8 +\b K)\bm\phi=0,$ i.e. the direct problem for $ e^{inc}=\partial_n e^{inc}=\partial_{\tau}e^{inc}=0,$ on $\Gamma_0$. From \autoref{theorem1}, we have that $e^j=h^j=0,$ in $\Omega_j,$ for  $ j= 0,1,2.$

We construct the fields
\begin{equation*}
\tilde{e}({\b x})=\mathcal{S}_{100}\psi_1^e({\b x})+\mathcal{S}_{101}\psi_2^e({\b x}),\quad \tilde{h}({\b x})=\mathcal{S}_{100}\psi_1^h({\b x})+\mathcal{S}_{101}\psi_2^h({\b x}),\quad \b x\in\Omega_0, 
\end{equation*}
which are radiating solutions of the Helmholtz equation in $\Omega_0.$ Thus,  $\tilde{e} =\tilde{h} =0,$ in $\Omega_0,$ and consequently on $\Gamma_0.$ The continuity of the single layer potential gives $ e_1 = \tilde{e} = 0,$ and $h_1 = \tilde{h} = 0,$ on $\Gamma_0.$ On the other hand, the jump-relation of its normal derivative across $\Gamma_0$  results to 
\begin{equation*}
0=NS_{100}\psi_1^u -\frac{1}{2}\psi_1^u +NS_{101}\psi_2^u =NS_{100}\psi_1^u +\frac{1}{2}\psi_1^u +NS_{101}\psi_2^u,  \quad \mbox{for } u= e,h.
\end{equation*}
Then $\psi_1^e = \psi_1^h = 0,$ on $\Gamma_0.$ The relation \eqref{identity}, gives also $\psi_0^e = \psi_0^h = 0.$ 

From the representation \eqref{solutions} and the jump-relations across the boundary $\Gamma_1$ we obtain
\begin{equation*}
\bigg(NS_{111}-\frac{1}{2}\bigg)\psi_2^e =0,\quad\bigg(NS_{111}-\frac{1}{2}\bigg)\psi_2^h =0, \quad \mbox{on   } \Gamma_1.
\end{equation*} 
Since $\kappa_1^2$ is not an interior Dirichlet eigenvalue in $\Omega_1$, the unique solvability of the above integral equations gives   $\psi_2^e =\psi_2^h=0$ on $\Gamma_1$. Again, using \eqref{identity}, we also get $\psi_3^e =\psi_3^h=0,$ on $\Gamma_1$.

As the fields $e^0, \, h^0$ tend to $\Gamma_0,$ using \eqref{homogenous1} and \eqref{homogenous3} we get
\begin{equation*}
\bigg(D_{000}-\frac{1}{2}\bigg)\phi_0^e =0,\quad \bigg(D_{000}-\frac{1}{2}\bigg)\phi_0^h =0,\quad \mbox{on   } \Gamma_0.
\end{equation*}
The injectivity here follows from the assumption that $\kappa_0^2$ is not an interior Dirichlet eigenvalue in $\mathbb{R}^2\setminus\overline{\Omega_0}.$ Then $\phi_0^e =\phi_0^h=0,$ on $\Gamma_0$.

The same procedure for the fields $e^2, \,h^2,$ reduce the boundary conditions \eqref{homogenous5} and \eqref{homogenous7} to
\begin{equation*}
\bigg(D_{211}-\frac{1}{2}\bigg)\phi_3^e =0,\quad \bigg(D_{211}-\frac{1}{2}\bigg)\phi_3^h =0,\quad \mbox{on  } \Gamma_1.
\end{equation*}
The assumption on  $\kappa_2^2$ leads to the trivial solution $\phi_3^e =\phi_3^h =0,$ on $\Gamma_1$. This completes the proof, since $\bm\phi=0$.
\end{proof}

\section{Numerical implementation}\label{sec_numerics}

We solve numerically the direct problem  \eqref{helm} -- \eqref{radiation} considering the solution of the linear system \eqref{system1} or \eqref{eq_system}. We handle the singularities of the kernels of the integral operators using quadrature rules and we approximate the smooth kernels with the trapezoidal rule \cite{Kre14}. We do not present here the forms and the decompositions of the kernels since they can be found in previous works, see for instance \cite{ColKre13, Kre14}. We address the Maue's formulas in order to reduce the hyper-singularity of the normal and tangential derivative of the double layer potential \cite{GinMin16}. We obtain a linear system by collocating the system of integral equations at the nodal points using trigonometric polynomial approximations \cite{Kre14}.  Standard convergence and error analysis applies in this case \cite{Kre95}.

We present results for two examples. In the first one, we consider four arbitrary point sources and we construct boundary data such that we have analytic fields as solutions. We compare them with the numerical solution. The expected exponential convergence is clearly achieved \cite{Kre90}. The second example deals with the initial scattering problem by an obliquely incident wave. The correctness of the derived solution cannot be checked for this case but the first example justifies the accuracy of the proposed numerical scheme. 

We assume the following parametric representation for the smooth boundary curves 
\begin{equation*}
\begin{aligned}
\Gamma_0 &:= \left\lbrace \b x^0(t)=(x_1^0(t),\, x_2^0(t)), \,\,  t\in[0,2\pi]\right\rbrace, \\
\Gamma_1 &:= \left\lbrace \b x^1(t)=(x_1^1(t),\, x_2^1(t)), \,\,  t\in[0,2\pi]\right\rbrace,
\end{aligned}
\end{equation*}
where $\b x^0, \, \b x^1:\mathbb{R}\rightarrow\mathbb{R}^2$ are $C^2$-smooth, $2\pi$-periodic, injective and counter-clockwise oriented parametrizations. We consider $2n$ equidistant collocation points
\begin{equation*}
t_j =\frac{j\pi}{n},\quad j=0,...,2n-1 .
\end{equation*}

\textbf{Example 1 (analytic solution)}  We consider four arbitrary points $\b z _1, \, \b z _2\in\Omega,$ and $\b z _3, \, \b z _4\in\Omega_0$ and  we define the boundary functions $f^k_j (\b x), \,k=0,1, \,  j=1,2,3,4,$ by

\begin{align*}
f^0_1( \b x ) &=H_0^{(1)}(\kappa_1| \b r _3( \b x )|)-H_0^{(1)}(\kappa_0| \b r _1( \b x )|), \\
f^0_2 ( \b x ) &=-\kappa_1\tilde{\mu}_1 \omega\dfrac{H_1^{(1)}(\kappa_1 | \b r _4( \b x )|) \, \b n ( \b x )\cdot \b r _4 ( \b x )}{| \b r _4( \b x )|} -\kappa_1 \beta_1\dfrac{H_1^{(1)}(\kappa_1 | \b r _3( \b x )|) \,{\bm{\tau}}( \b x )\cdot \b r _3 ( \b x )}{| \b r _3( \b x )|} \\
&\phantom{=}+\kappa_0\tilde{\mu_0}\omega \dfrac{H_1^{(1)}(\kappa_0 | \b r _2( \b x )|) \, \b n ( \b x )\cdot \b r _2 ( \b x )}{| \b r _2( \b x )|}+\kappa_0\beta_0 \dfrac{H_1^{(1)}(\kappa_0 | \b r _1( \b x )|) \,{\bm{\tau}}( \b x )\cdot \b r _1 ( \b x )}{| \b r _1( \b x )|}, \\
f^0_3 ( \b x ) &=H_0^{(1)}(\kappa_1| \b r _4( \b x )|)-H_0^{(1)}(\kappa_0| \b r _2( \b x )|), \\
f^0_4 (\b x ) &=-\kappa_1\tilde{\epsilon_1}\omega\dfrac{H_1^{(1)}(\kappa_1 | \b r _3( \b x )|) \,\b n ( \b x )\cdot \b r _3 ( \b x )}{| \b r _3( \b x )|}+\kappa_1 \beta_1\dfrac{H_1^{(1)}(\kappa_1 |\b r _4( \b x )|) \, {\bm{\tau}}( \b x )\cdot \b r _4 (\b x )}{| \b r _4(\b x )|} \\
&\phantom{=}+\kappa_0\tilde{\epsilon_0}\omega \dfrac{H_1^{(1)}(\kappa_0 | \b r _1( \b x )|) \, \b n ( \b x )\cdot \b r _1 ( \b x )}{| \b r _1( \b x )|)} -\kappa_0\beta_0 \dfrac{H_1^{(1)}(\kappa_0 | \b r _2( \b x )|) \, {\bm{\tau}}( \b x )\cdot \b r _2 ( \b x )}{| \b r _2( \b x )|}, \\
f^1_1( \b x ) &=H_0^{(1)}(\kappa_1| \b r _3( \b x )|)-H_0^{(1)}(\kappa_2| \b r _3( \b x )|), \\
f^1_2( \b x ) &=-\kappa_1\tilde{\mu_1}\omega\dfrac{H_1^{(1)}(\kappa_1 | \b r _4( \b x )|) \, \b n ( \b x)\cdot \b r _4 ( \b x )}{| \b r _4( \b x )|}-\kappa_1 \beta_1\dfrac{H_1^{(1)}(\kappa_1 | \b r _3( \b x )|) \, {\bm{\tau}}( \b x )\cdot \b r _3 ( \b x )}{| \b r _3( \b x )|} \\
&\phantom{=}-\kappa_2\tilde{\mu_2}\omega \dfrac{H_1^{(1)}(\kappa_2 | \b r _4( \b x )|) \, \b n ( \b x )\cdot \b r _4 ( \b x )}{| \b r _4( \b x )|} -\kappa_2\beta_2 \dfrac{H_1^{(1)}(\kappa_2 | \b r _3( \b x )|) \,{\bm{\tau}}( \b x )\cdot \b r _3 ( \b x )}{| \b r _3( \b x )|}, \\
f^1_3(\b x ) &=H_0^{(1)}(\kappa_1 | \b r _4( \b x )|)-H_0^{(1)}(\kappa_2| \b r _4( \b x )|), \\ 
f^1_4 ( \b x ) &=-\kappa_1\tilde{\epsilon_1}\omega\dfrac{H_1^{(1)}(\kappa_1 | \b r _3( \b x )|) \, \b n ( \b x )\cdot \b r _3 ( \b x )}{| \b r _3( \b x )|} +\kappa_1 \beta_1\dfrac{H_1^{(1)}(\kappa_1 | \b r _4( \b x )|)\,{\bm{\tau}}( \b x )\cdot \b r _4 ( \b x )}{| \b r _4( \b x )|} \\
&\phantom{=}-\kappa_2\tilde{\epsilon_2}\omega \dfrac{H_1^{(1)}(\kappa_2 | \b r _3( \b x )|) \, \b n ( \b x )\cdot \b r _3 (\b x)}{|\b r _3( \b x )|} +\kappa_2\beta_2 \dfrac{H_1^{(1)}(\kappa_2 | \b r _4( \b x )|) \, {\bm{\tau}}( \b x )\cdot \b r _4 ( \b x )}{| \b r _4( \b x )|},
\end{align*}
where $ \b r _j( \b x )= \b x - \b z _j, \, j = 1,2,3,4.$

Then, the fields 
\begin{equation}\label{exact_sol}
\begin{aligned}
e^{0}( \b x ) &=H_0^{(1)}(\kappa_0 | \b r _1( \b x )|), & h^{0}( \b x ) &=H_0^{(1)}(\kappa_0 | \b r_2( \b x )|), & \b x &\in\Omega_0, \\
e^{1}( \b x ) &=H_0^{(1)}(\kappa_1 | \b r _3( \b x )|), & h^1 ( \b x ) &=H_0^{(1)}(\kappa_1 | \b r_4 ( \b x )|), & \b x &\in\Omega_1, \\
e^{2}( \b x ) &=H_0^{(1)}(\kappa_2 | \b r _3( \b x )|), & h^2 ( \b x ) &=H_0^{(1)}(\kappa_2 | \b r_4 ( \b x )|), & \b x &\in\Omega_2, 
\end{aligned}
\end{equation}
satisfy the system of Helmholtz equations
\[
\Delta e^j + \kappa_j^2 e^j =0, \quad \Delta h^j + \kappa_j^2 h^j =0, \quad j = 0,1,2,
\]
and the transmission boundary conditions
\begin{equation*}
\begin{aligned}
f^0_1 &= e^1 - e^0, & \mbox{on  } \Gamma_0, \\
f^0_2 &= \tilde{\mu}_1\omega\dfrac{\partial h^1}{\partial n}+{\beta_1}\dfrac{\partial e^1}{\partial\tau}-\tilde{\mu}_0\omega\dfrac{\partial h^{0}}{\partial n}-\beta_0\dfrac{\partial e^{0}}{\partial\tau}, & \mbox{on  } \Gamma_0, \\
f^0_3 &= h^1 - h^0, & \mbox{on  } \Gamma_0, \\
f^0_4 &= \tilde{\epsilon}_1\omega\dfrac{\partial e^1}{\partial n}-{\beta_1}\dfrac{\partial h^1}{\partial\tau}-\tilde{\epsilon}_0\omega\dfrac{\partial e^{0}}{\partial n}+{\beta_0}\dfrac{\partial h^{0}}{\partial\tau},  & \mbox{on  } \Gamma_0, \\
f^1_1 &=e^1 - e^2, & \mbox{on  } \Gamma_1, \\
f^1_2 &=\tilde{\mu}_1\omega\dfrac{\partial h^1}{\partial n}+{\beta_1}\dfrac{\partial e^1}{\partial\tau}-\tilde{\mu}_2\omega\dfrac{\partial h^2 }{\partial n}-\beta_2 \dfrac{\partial e^2 }{\partial\tau}, & \mbox{on  } \Gamma_1, \\
f^1_3 &=h^1-h^2, & \mbox{on  } \Gamma_1, \\
f^1_4 &=\tilde{\epsilon}_1\omega\dfrac{\partial e^1}{\partial n}-{\beta_1}\dfrac{\partial h^1}{\partial\tau}-\tilde{\epsilon}_2\omega\dfrac{\partial e^2 }{\partial n}+{\beta_2}\dfrac{\partial h^2 }{\partial\tau},  & \mbox{on  } \Gamma_1.\\
\end{aligned}
\end{equation*}
The exterior fields $e^0, \, h^0$ satisfy in addition the radiation condition \eqref{radiation}. 

\begin{figure}[t]
\begin{center}
\includegraphics[scale=0.9]{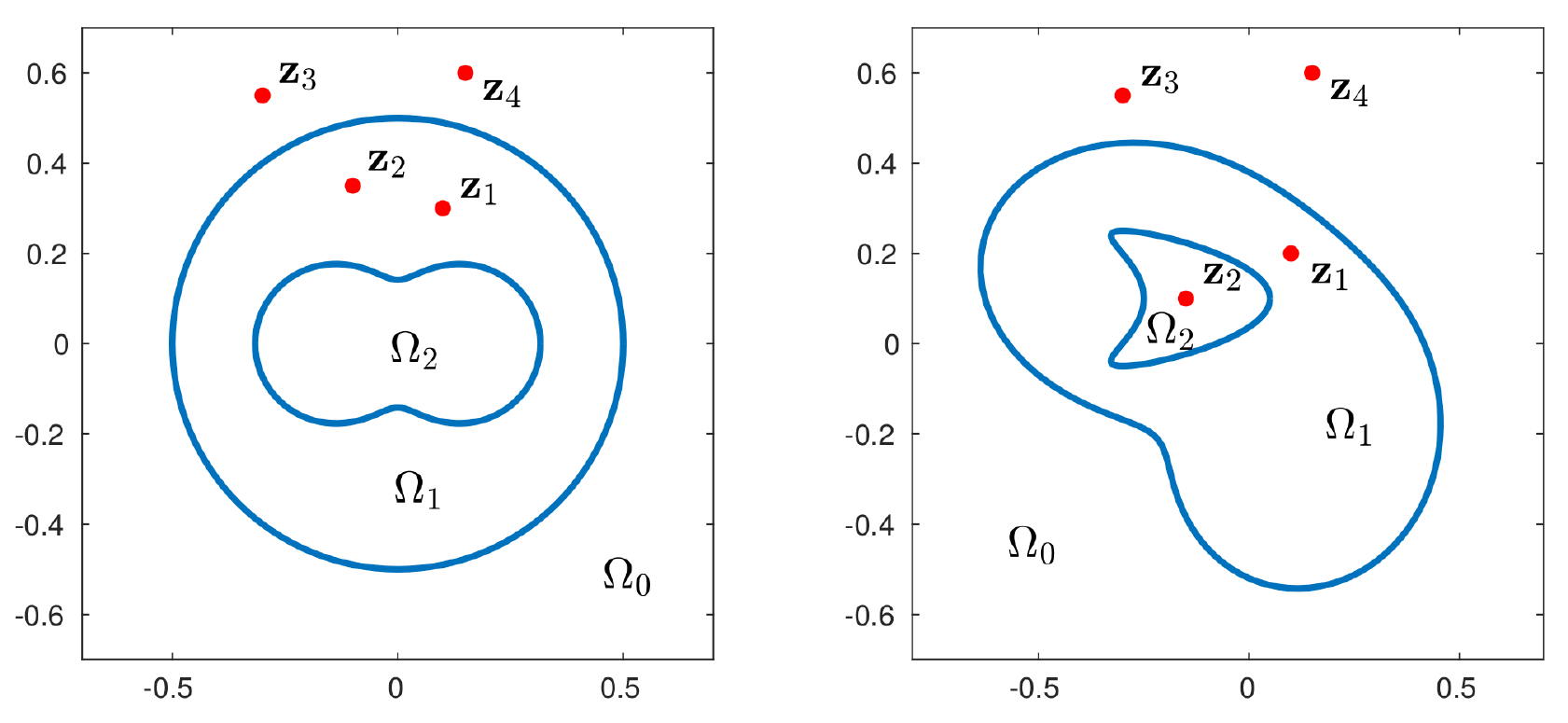}
\caption{ The geometry of the problem and the position of the point sources considered in the first (left) and in the second (right) case of the first example.  }\label{fig1}
\end{center}
\end{figure}

We compare the numerical solutions $u^j_n,$ for $u= e,h, \, j= 0,1,2,$ with 
the exact solutions \eqref{exact_sol}, with respect to the discretization parameter $n.$ Using the asymptotic behavior of the Hankel function, we can correlate also the exact far-field of the scattered wave, given by
\begin{equation*}
e^{\infty}({\bf{\hat{x}}})=\dfrac{-4ie^{i\pi/4}}{\sqrt{8\pi\kappa_0}}e^{-i\kappa_0\bf{\hat{x}}\cdot \b{z}_1},\quad h^{\infty}({\bf{\hat{x}}})=\dfrac{-4ie^{i\pi/4}}{\sqrt{8\pi\kappa_0}}e^{-i\kappa_0\bf{\hat{x}}\cdot {\bf{z}_2}},\quad{\bf{\hat{x}}}\in S,
\end{equation*}
with the numerical one which takes the form
\begin{equation*}
\begin{aligned}
e_n^{\infty}({\bf{\hat{x}}}(t)) &=\dfrac{e^{i\pi/4}}{\sqrt{8\pi\kappa_0}}\int_0^{2\pi}e^{-i\kappa_0{\bf{\hat{x}}}\cdot{\b {x}}^0(t)}\left[-i\kappa_0({\bf{\hat{x}}}\cdot{\bf{n}}({\bf{x}}^0(t)))\phi_0^e(t)-\psi_0^e(t)\right]|{\bf{x}}^{0'}(t)|dt,
\\
h_n^{\infty}({\bf{\hat{x}}}(t)) &=\dfrac{e^{i\pi/4}}{\sqrt{8\pi\kappa_0}}\int_0^{2\pi}e^{-i\kappa_0{\bf{\hat{x}}}\cdot{\bf{x}}^0(t)}\left[-i\kappa_0({\bf{\hat{x}}}\cdot{\bf{n}}({\bf{x}}^0(t)))\phi_0^h(t)-\psi_0^h(t)\right]|{\bf{x}}^{0'}(t)|dt,
\end{aligned}
\end{equation*}
considering the representation \eqref{solutions}, where now the density functions solve \eqref{system1} with the right-hand side replaced by
$\b f = (f^0_1 ,\, ... \,, f^0_4, \, f^1_1 , \,... \,, f^1_4)^\top.$

{\renewcommand{\arraystretch}{1.4}
 \begin{table}[t]
\begin{center}
 \begin{tabular}{| c  | c  | c  | } 
 \hline
 $n$ & $ e_n^1 (0, \,- 0.3) $ & $ h_n^1 (0, \,-0.3) $  
\\ \hline 
8 & $ 0.358002472423  +i\,    0.465413341071  $ & $   0.340975378878  + i \,    0.470609286666     $\\
16 & $ 0.371293663181  +i\,       0.464362535355 $ & $    0.359126522592 +i\,       0.469698417374       $\\
32 & $ 0.371625795959   +i\,      0.464351276823  $ & $   0.359231933566  +i\,      0.469693507804       $\\
 64   & $    0.371625444291   +i\,      0.464351267787  $ & $   0.359232825162  +i\,      0.469693536105       $\\
 \hline
\multicolumn{1}{c|}{} &  $ e^1 (0, \, -0.3) $ & $ h^1 (0, \,-0.3) $   \\
\cline{2-3} 
\multicolumn{1}{c|}{} &  $    0.371625444291  +i\,       0.464351267787  $ & $    0.359232825161 +i\,       0.469693536105    $  \\ \cline{2-3}
\end{tabular}
\caption{The computed and the exact interior electric and magnetic fields in $\Omega_1.$  }\label{table1} 
\bigskip
 \begin{tabular}{| c  | c  | c  | } 
 \hline
 $n$ & $ e_n^2 (0.2, \, 0) $ & $ h_n^2 (0.2, \,0) $  
\\ \hline 
8 & $ 0.095999847542   +i\,      0.523746820044  $ & $   0.334044441714   +i\,     0.476243803875      $\\
16 & $ 0.110704266109    +i\,     0.520805574336  $ & $   0.350976468890   +i\,     0.473070152121     $\\
 32    & $   0.111073699187    +i\,     0.520785772802  $ & $   0.351044736683   +i\,     0.473067712107     $\\
 64     & $  0.111073338024   +i\,      0.520785762600  $ & $   0.351045738875   +i\,     0.473067595167          $\\
 \hline
\multicolumn{1}{c|}{} &  $ e^2 (0.2, \, 0) $ & $ h^2 (0.2, \,0) $   \\
\cline{2-3} 
\multicolumn{1}{c|}{} &  $    0.111073338024    +i\,     0.520785762600  $ & $   0.351045738874  +i\,      0.473067595167      $  \\ \cline{2-3}
\end{tabular}
\caption{The computed and the exact interior electric and magnetic fields in $\Omega_2.$ }\label{table2} 
\end{center}
\end{table}

We consider a peanut-shaped  interior boundary $\Gamma_1$ with parametric form
\[
\b x^1 (t) = \sqrt{0.1 \cos^2 t + 0.02 \sin^2 t}\, (\cos t, \, \sin t), \quad t \in [0,2\pi],
\]
and the $\Gamma_0$ is a circle with center $(0,\,0)$ and radius $0.5.$ The material parameters are $(\epsilon_0, \, \mu_0) = (1,\,1),$ $(\epsilon_1, \, \mu_1) = (2,\,2),$  and $(\epsilon_2, \, \mu_2) = (3,\,3).$ We set $\omega =1$ and $\theta = \pi/3.$ The source points are located at the positions $\b z_1 = (0.1, \, 0.3), \, \b z_2 = (-0.1,\, 0.35) \in \Omega_1,$ and $\b z_3 = (-0.3, \, 0.55), \, \b z_4 = (0.15,\, 0.6) \in \Omega_0,$ see the left picture in \autoref{fig1}.

In \autoref{table1} and \autoref{table2} we see the numerical and the exact values of the interior fields at the position $(0,\, -0.3) \in \Omega_1,$ and $(0.2,\, 0) \in \Omega_2,$ respectively, for increasing discretization number $n.$ The comparison between the numerical and the exact scattered fields at the near- and the far-field is presented in \autoref{table3} and \autoref{table4}, respectively.  We compute the near-field at the position $(0.2, \, 0.7) $ and the far-field at the direction $\b{\hat{x}}( 0).$ The exponential convergence is clearly exhibited, as we see also in \autoref{fig2} where we plot the $L^2$-norm (in semi-logarithmic scale) of the difference between the exact and the computed near- and far-fields, respectively.

The numerical results are independent of the parametrization of the boundary and of the material parameters. To support that, we consider also the following case: a kite-shaped  interior boundary $\Gamma_1$ with parametric form
\[
\b x^1 (t) =  (0.15\cos t +0.1\cos 2t - 0.2, \, 0.15\sin t + 0.15), \quad t \in [0,2\pi],
\]
and an apple-shaped boundary $\Gamma_0$ having the parametrization
\[
\b x^0 (t) = \frac{0.45 + 0.3 \cos t -0.1 \sin 2t}{1+0.7 \cos t} \, (\cos t, \, \sin t), \quad t \in [0,2\pi].
\]
We use $\omega=2$ and $\theta = \pi/6,$ with material parameters $(\epsilon_0, \, \mu_0) = (1,\,1),$ $(\epsilon_1, \, \mu_1) = (3,\,2),$  and $(\epsilon_2, \, \mu_2) = (4,\,3).$ The locations of the source points are given in the right picture of \autoref{fig1}. The expected convergence is obtained also for this case, as \autoref{fig3} demonstrates.

\begin{figure}[t]
\begin{center}
\includegraphics[scale=0.9]{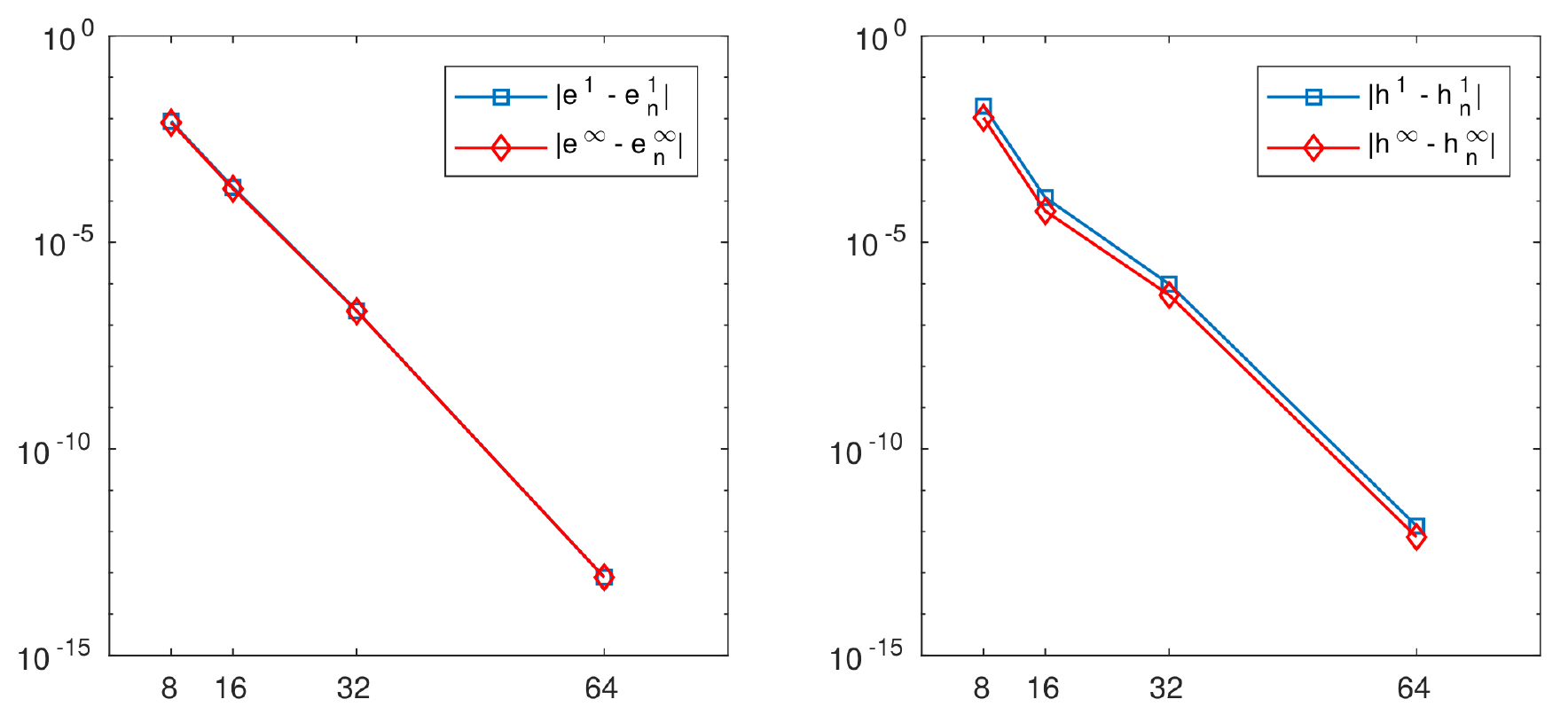}
\caption{The $L^2-$norm (in semi-logarithmic scale) of the difference between the computed and the exact interior (blue line) and the far-field (red line) of the electric (left) and the magnetic (right) fields. The plots are with respect to $n,$ for the first case of the first example.  }\label{fig2}
\bigskip
\includegraphics[scale=0.9]{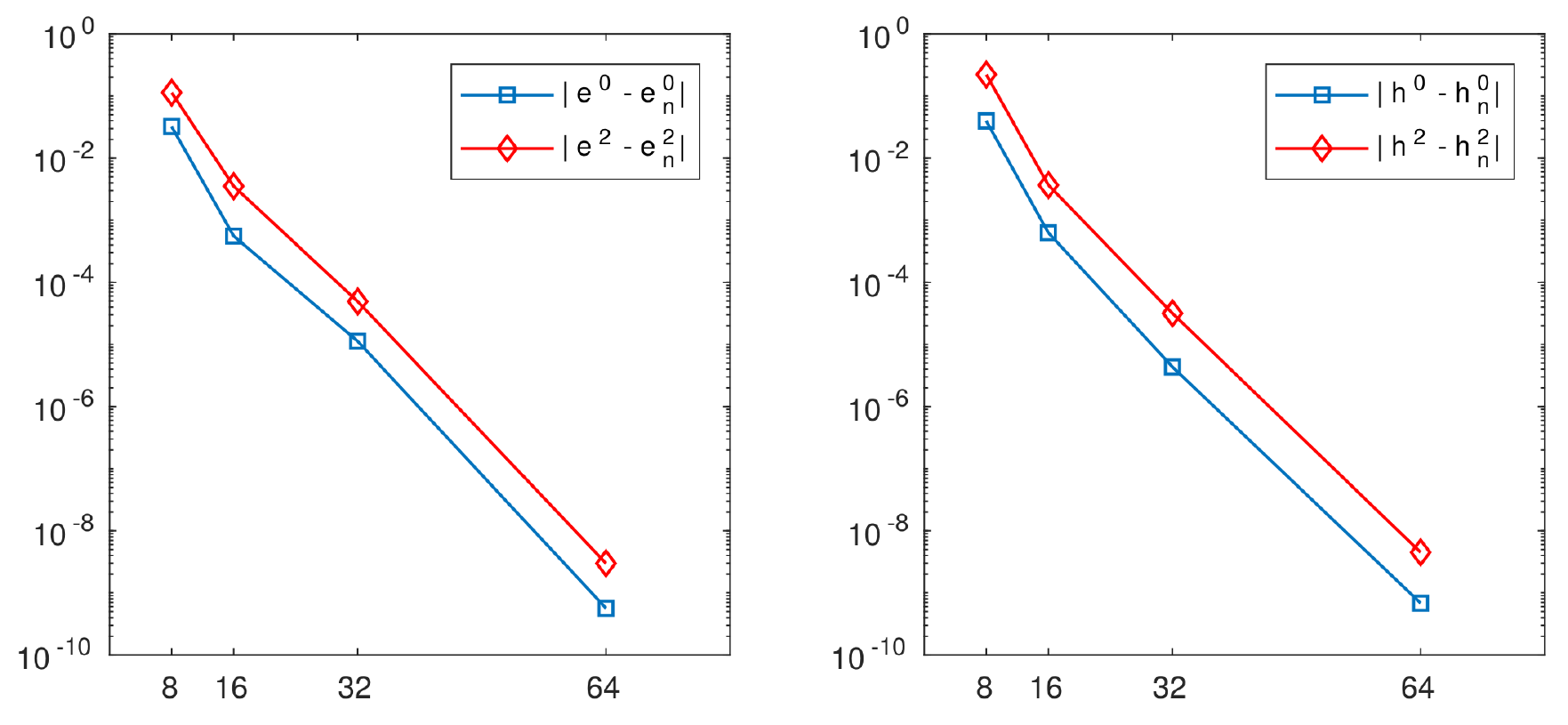}
\caption{ The $L^2-$norm (in semi-logarithmic scale) of the difference between the computed and the exact scattered (blue line) and the interior (red line) electric (left) and the magnetic (right) fields. The plots are with respect to $n,$ for the second case of the first example.  }\label{fig3}
\end{center}
\end{figure}

{\renewcommand{\arraystretch}{1.4}
 \begin{table}[t]
\begin{center}
 \begin{tabular}{| c  | c  | c  | } 
 \hline
 $n$ & $ e_n^0 (0.2, \, 0.7) $ & $ h_n^0 (0.2, \,0.7) $  
\\ \hline 
8 & $ 0.959974539981      -i\, 0.688119559604  $ & $    0.949222323693     -i\, 0.596002286134      $\\
 16  & $     0.968181135721      -i\, 0.686366381756   $ & $   0.960497156194     -i\, 0.607397006008      $\\
 32  & $     0.968378330983      -i\, 0.686291758194 $ & $     0.960550801453     -i\,0.607417933832      $\\
 64   & $    0.968378106099      -i\, 0.686291722922  $ & $    0.960551378447     -i\, 0.607418413636                    $\\
 \hline
\multicolumn{1}{c|}{} &  $ e^0 (0.2, \, 0.7) $ & $ h^0 (0.2, \,0.7) $   \\
\cline{2-3} 
\multicolumn{1}{c|}{} &  $      0.968378106099    -i\, 0.686291722922  $ & $    0.960551378446   -i\, 0.607418413635     $  \\ \cline{2-3}
\end{tabular}
\caption{The computed and the exact scattered electric and magnetic fields in $\Omega_0.$ }\label{table3} 
\bigskip
 \begin{tabular}{| c  | c  | c  | } 
 \hline
 $n$ & $ e_n^\infty  (\b{\hat{x}}( 0)) $ & $ h_n^\infty (\b{\hat{x}}( 0)) $  
\\ \hline 
8 & $ 0.542584437477      -i\,0.654900164839 $ & $     0.648825689486     -i\,0.546227842661     $\\
 16  & $     0.551320002791      -i\,0.656421849521  $ & $    0.656410781814     -i\,0.551518614562     $\\
 32   & $    0.551551141020      -i\,0.656427307849  $ & $    0.656426811866     -i\,0.551550742378     $\\
 64  & $     0.551550951838      -i\,0.656427255240 $ & $     0.656427255241     -i\,0.551550951838     $\\
 \hline
\multicolumn{1}{c|}{} &  $ e^\infty (\b{\hat{x}}( 0)) $ & $ h^\infty (\b{\hat{x}}( 0)) $   \\
\cline{2-3} 
\multicolumn{1}{c|}{} &  $         0.551550951838      -i\,0.656427255240  $ & $     0.656427255240     -i\, 0.551550951838         $  \\ \cline{2-3}
\end{tabular}
\caption{The computed and the exact far-field of the electric and magnetic fields. }\label{table4} 
\end{center}
\end{table}

\textbf{Example 2 (oblique incidence)} We consider the scattering problem of an obliquely incident wave of the form \eqref{incident}, for different values of the polar angle $\phi,$ which corresponds to the incident direction in $\R^2.$ For the setup of the first example, with $\omega = 2$ and $\phi = \pi/6,$ we present the distribution of the norms $|e^j_n |$ and $|h^j_n |,$ for $j=0,1,2,$ in \autoref{fig4}. The values in \autoref{fig5}, correspond to the second case for $\omega = 2$ and $\phi = \pi/2.$ The material parameters are kept the same as in Example 1.

\begin{figure}[t]
\begin{center}
\includegraphics[scale=0.8]{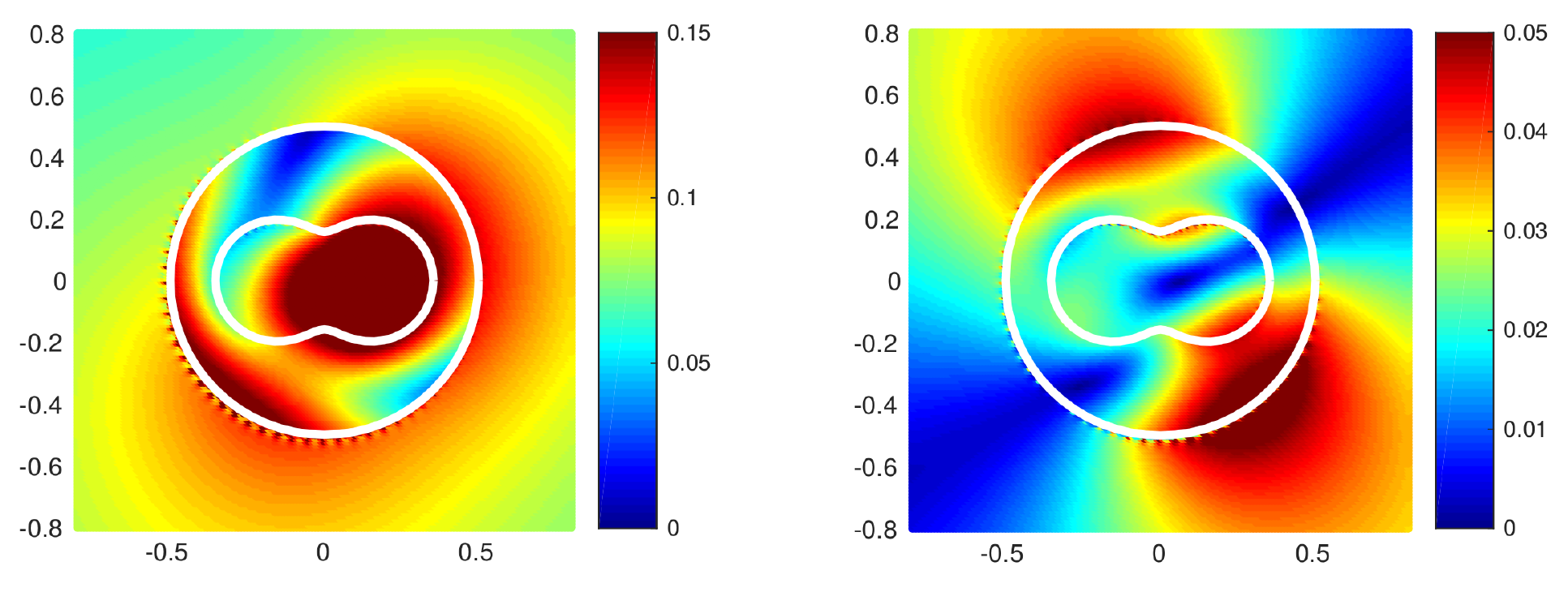}
\caption{ The norm of the electric (left) and magnetic (right) field, for $\omega =2$ and $\phi = \pi/6.$ }\label{fig4}
\bigskip
\includegraphics[scale=0.82]{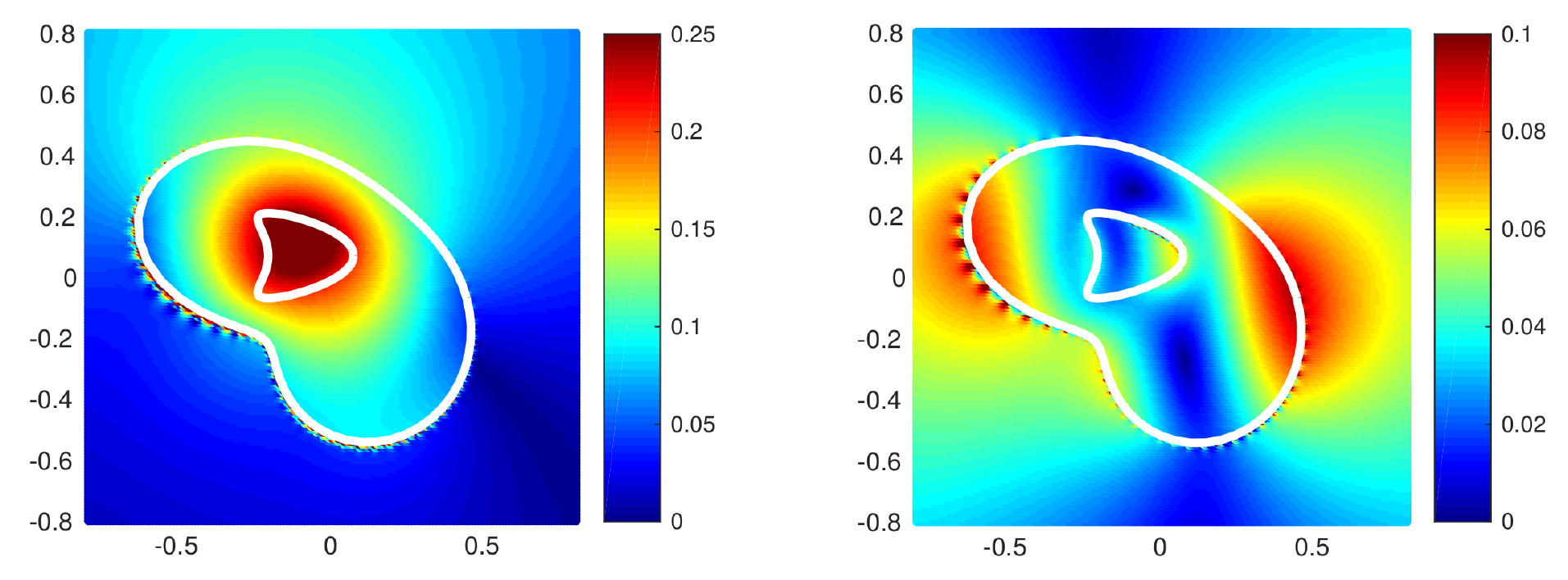}
\caption{ The norm of the electric (left) and magnetic (right) field, for $\omega =2$ and $\phi = \pi/2.$ }\label{fig5}
\end{center}
\end{figure}

\section{Conclusions}

In this work we addressed the scattering problem of a time-harmonic electromagnetic wave by an infinitely long, piecewise constant inhomogeneous and penetrable cylinder. The incident wave is transverse magnetic polarized. The 3D direct problem can be reduced to a 2D problem and we examined its well-posedness. The complexity of the problem is reflected in the transmission boundary conditions where the normal and tangential derivatives of the fields are coupled. We proved that the direct problem is equivalent to an interior eigenvalue problem for an elliptic and properly elliptic operator and we showed that the Shapiro-Lopatinskij condition is satisfied on both boundaries.  Thus, we obtained uniqueness and existence followed from the integral equation method.  Using a special integral representation of the fields, we derived a convergent scheme for the numerical approximation of the solution.

\section*{Acknowledgements}
The work of SG was co-financed by Greece and the European Union (European Social Fund-ESF) through the Operational Program $``$Human Resources Development, Education and Lifelong Learning$"$ in the context of the project $``$Strengthening Human Resources Research Potential via Doctorate Research-2nd Cycle$"$ (MIS-5000432), implemented by the State Scholarships Foundation (IKY).
The work of LM was supported by the Austrian Science Fund (FWF) in the project F6801--N36 within the Special Research Program SFB F68 $``$Tomography Across the Scales$".$

\appendix
\section{The Shapiro-Lopatinskij condition}
\label{Appendix}
In this section, we summarize the results from \cite{RoitShef,Wloka,WloRowLaw} needed for the proof of the Shapiro-Lopatinskij condition. The notation follows that of the referred works.  
 
Let $\Omega$ be an open bounded set in $\mathbb{R}^d$, with boundary $\Gamma_0$ and let $\Omega_2$ be a subdomain of $\Omega$ with boundary $\Gamma_1$ disjoint from $\Gamma_0$. Also, we set $\Omega_1:=\Omega\setminus\overline{\Omega_2}$, where $\overline{\Omega_2}=\Omega_2\cup\Gamma_1$. 
We consider the following boundary value problem
\begin{equation*}
\begin{aligned}
\b {A}^1(\b x,D) \b u^1(\b x) &= 0, & \b x\in\Omega_1, \\
\b {A}^2(\b x,D) \b u^2(\b x) &= 0,  & \b  x\in\Omega_2, \\
\b B^1(\b  y,D)  \b u^1(\b y) &= 0,  &\b  y\in\Gamma_0, \\
\b B^2(\b  y,D)\b u^1(\b y)+\b B^3(\b y,D)\b u^2(\b  y) &= 0, & \b  y\in\Gamma_1, \\
\end{aligned}
\end{equation*}
where $\b u^r \in H^2(\Omega_r)$ are vectors of size $b$, for $r=1,2$.
The linear partial differential operator $\b A^r$ is a  $b\times b$ matrix-valued operator defined by
\begin{equation*}
\b A^r(\b x,D):=\sum_{|s|\leq m}a_s^{r}(\b x)D^s,\quad \b x\in\overline{\Omega}_r,
\end{equation*}
where $m$ is the order of the differential operator and $a_s^r$ are smooth coefficients, with $D^s=D_1^{s_1}...D_d^{s_d}$, $D_j=i^{-1}\partial/\partial x_j$  and $|s|=s_1+...+s_d$. 

The  boundary differential operator is a $l\times b$ matrix-valued operator  given by
\begin{equation}\label{boundoper}
\begin{aligned}
\b B^c(\b y,D)=\sum_{|s|\leq m'}b_s^c(\b y) T_0(D^s), \quad \b y\in\Gamma_p, 
\end{aligned}
\end{equation}
where  $T_0$ is the trace operator, $m'$ denotes the order of the differential operator, with $b_s^c$ are smooth coefficients  for $p=0,1$ and $c=1,2,3$.

First, we show that the operators $\b A^1 $ and $\b B^1 $ satisfy the Shapiro-Lopatinskij condition on $\Gamma_0$. Using Fourier transformation, we may transform the boundary value problem 
\begin{equation*}
\begin{aligned}
\b A^1 (\b x,D)\b u^1 (\b x) &= 0, &\b x\in\Omega_1, \\
\b B^1 (\b y,D)\b u^1 (\ y) &= 0, &\b y\in\Gamma_0 ,
\end{aligned}
\end{equation*}
 to an initial value problem.
 
 We consider the case of $\b A=\b A^1=\b A^2.$ Let $\b x_0\in\Gamma_0$,  and $\b A_0$ be the principal part of the matrix $\b A$. We set $\b x_0$ to be at the origin of the coordinate system  and we choose the coordinate axis $x_d$ in the direction of the inward pointing normal and the other coordinates are perpendicular to $x_d$. Using the Fourier transform
\begin{equation*}
\mathcal{F}\{f\}(\b q) :=\int_{\mathbb{R}^d}e^{-i(\b z,\, \b q)}f(\b z)d\b z, \quad f\in L_1(\mathbb{R}^d),
\end{equation*}
the basic derivatives $ (1/i) \partial/\partial x_1 ,...,(1/i) \partial/\partial x_{d-1}$ (equipped with the factor $1/i$) are transformed to $(\xi_1,...,\xi_{d-1})=\bm \xi'\in T_{\Gamma_0}$, where $T_{\Gamma_0}$ is the tangential hyperplane of $\Gamma_0$ at the point $\b x_0$ \cite{Wloka}.
Setting $x_d=t$, we have
\begin{equation*}
\mathcal{F}_{d-1}\b A_0(\b x_0,D)=\b A_0\bigg(\b x_0;\bm\xi',\frac{1}{i}\frac{\partial}{\partial x_d}\bigg)=\b A_0\bigg(\b x_0;\bm\xi',\frac{1}{i}\frac{\partial}{\partial t}\bigg),
\end{equation*}
for $\bm\xi'\neq 0$. We consider the linear ordinary differential equation with constant coefficients
\begin{equation}\label{initprob}
\b A_0\bigg(\b x_0;\bm\xi',\frac{1}{i}\frac{\partial}{\partial t}\bigg)\b u^1(t)=0,\quad t>0,\,\, 0\neq\bm\xi'\in T_{\Gamma_0}.
\end{equation}
The solution space $\mathcal{M}$ of \eqref{initprob} decomposes to the direct sum
\begin{equation*}\label{solspace}
\mathcal{M}=\mathcal{M}^+\oplus\mathcal{M}^0\oplus\mathcal{M}^-,
\end{equation*}
where $\mathcal{M}^+$ and $\mathcal{M}^-$ are the solution spaces for the roots of $\mbox{det} P(\lambda)=\mbox{det}\b A_0(\b x_0,\xi ',\lambda)$, which are in the upper half plane $\Im \lambda>0$ and in the lower half plane $\Im \lambda<0,$ respectively. The solution space $M^0=\{0\}$, since we have assumed that the operator $\b {A}$ is elliptic i.e. $\mbox{det}P(\lambda)$ has no roots on the real axis \cite{Wloka}. 

Let $\b B_{0}^c$ denote the principal part of the matrix $\b B^c$, for $c=1,2,3$. Similarly, using Fourier transform, we rewrite the initial value conditions for $t=0,$ as
\[
\b B_{0}^1\bigg(\b x_0;\bm \xi',\frac{1}{i}\frac{\partial}{\partial t}\bigg)\b u^1(t)\bigg|_{t=0}=0.
\]

Next, we formulate the Shapiro-Lopatinskij condition for homogeneous boundary conditions and we describe different ways to prove it. 

\begin{definition}[see \cite{Wloka}]
The pair of operators $\b{A}(\b y,D),\, \b{B}^1(\b y,D),$ for $\b y\in\Gamma_0$ is said to fulfill the Shapiro-Lopatinskij condition on $\Gamma_0$, if the following statement holds for all $\b y\in\Gamma_0$ and  $0\neq\bm\xi'\in T_{\Gamma_0}$. The homogeneous initial value problem
\begin{equation}\label{homogenous}
\begin{aligned}
\b A_0\bigg(\b y;\bm \xi'.\frac{1}{i}\frac{\partial}{\partial t}\bigg)\b u^1(t)&= 0, \quad t>0, \\
\b B_{0}^1\bigg(\b y;\bm\xi',\frac{1}{i}\frac{\partial}{\partial t}\bigg)\b u^1(t)\bigg|_{t=0} &=0,
\end{aligned}
\end{equation}
has in $\mathcal{M}^+$ the unique solution $\b u^1(t)=0$.
\end{definition}

\begin{theorem}[see \cite{WloRowLaw}] \label{equiv}
Let the operator $\b {A}(\b x,D)$ be properly elliptic and $\b {B}^1(\b y,D)$ be the boundary operator as in \eqref{boundoper}. We fix $\b y\in\Gamma_0,$ and $0\neq \bm \xi'\in T_{\Gamma_0}$. Then, the following statements are equivalent:
\begin{enumerate}
  \item The initial value problem \eqref{homogenous} has a unique solution.
  \item Let $a^+(\lambda)$ and $a^- (\lambda)$ denote the polynomial which contains all the roots above and below the real axis, respectively. Then, $\det \b A_0(\b x_0,\bm\xi ',\lambda)=a^+(\lambda)\,  a^-(\lambda).$ If $\b A_{co}$ denotes the cofactor matrix of $\b A_0$, then the rows of the matrix $\b B_{0}^1 \b A_{co}$ are linearly independent modulo $a^+(\lambda)$.
\end{enumerate}
\end{theorem}

Following \cite{Agran79}, we can apply this theory also to our case with the transmission boundary condition and show the equivalence to an initial boundary value problem. Then, the Shapiro-Lopatinskij condition is satisfied if  $\b B_{0}^2\b A_{co}\equiv 0(\mbox{mod}\,a^-)$ and $\b B_0^3\b A_{co}\equiv 0(\mbox{mod}\,a^+)$ \cite{RoitShef}.

\bibliographystyle{plain}
\bibliography{gindmindgiog_ref}

\begin{thebibliography}{10}

\bibitem{Agran79}
M.~S. Agranovich, Y.~V. Egorov, and M.~A. Shubin.
\newblock {\em Partial Differential Equations IX}, volume~79.
\newblock Springer-Encyclopedia of Mathematical Sciences, Berlin, 1997.

\bibitem{CakKre17}
F.~Cakoni and R.~Kress.
\newblock A boundary integral equation method for the transmission eigenvalue
  problem.
\newblock {\em Appl. Analysis}, 96(1):23--38, 2017.

\bibitem{ColKre83}
D.~Colton and R.~Kress.
\newblock {\em Integral equation methods in scattering theory}.
\newblock Classics in Applied Mathematics. Society for Industrial and Applied
  Mathematics, New York, 1983.

\bibitem{ColKre13}
D.~Colton and R.~Kress.
\newblock {\em Inverse Acoustic and Electromagnetic Scattering Theory}.
\newblock Number~93 in Applied Mathematical Sciences. Springer, New York, 3rd
  edition, 2013.

\bibitem{CosSte85}
M.~Costabel and E.~Stephan.
\newblock A direct boundary integral equation method for transmission problems.
\newblock {\em Journal of Mathematical Analysis and Applications},
  106:205--220, 1985.

\bibitem{GinMin16}
D.~Gintides and L.~Mindrinos.
\newblock The direct scattering problem of obliquely incident electromagnetic
  waves by a penetrable homogeneous cylinder.
\newblock {\em Journal of Integral Equations and Applications}, 28(1):91--122,
  2016.

\bibitem{Hsi11}
G.~C. Hsiao and L.~Xu.
\newblock A system of boundary integral equations for the transmission problem
  in acoustics.
\newblock {\em J. Comput. Appl. Math.}, 61:1017--1029, 2011.

\bibitem{KleMar88}
R.~E. Kleinman and P.~A. Martin.
\newblock On single integral equations for the transmission problem of
  acoustics.
\newblock {\em SIAM Journal on Applied Mathematics}, 48(2):307--325, 1988.

\bibitem{Kre90}
R.~Kress.
\newblock Numerical solution of boundary integral equations in time-harmonic
  electromagnetic scattering.
\newblock {\em Electromagnetics}, 10(1--2):1--20, 1990.

\bibitem{Kre95}
R.~Kress.
\newblock On the numerical solution of a hypersingular integral equation in
  scattering theory.
\newblock {\em SIAM Journal on Applied Mathematics}, 61(3):345--360, 1995.

\bibitem{Kre14}
R.~Kress.
\newblock {\em Linear Integral Equations}.
\newblock Springer, New York, 3rd edition, 2014.

\bibitem{Lee16}
S.~C. Lee.
\newblock Scattering at oblique incidence by multiple cylinders in front of a
  surface.
\newblock {\em J. Quant. Spectrosc. Ra.}, 182:119--127, 2016.

\bibitem{LucPanSche10}
M.~Lucido, G.~Panariello, and F.~Schettino.
\newblock Scattering by polygonal cross-section dielectric cylinders at oblique
  incidence.
\newblock {\em IEEE Transactions on Antennas and Propagation}, 58(2):540--551,
  2010.

\bibitem{Min18}
L.~Mindrinos.
\newblock The electromagnetic scattering problem by a cylindrical doubly
  connected domain at oblique incidence: the direct problem.
\newblock {\em IMA J. Appl. Math.}, 84:292--311, 2019.

\bibitem{NakWan13}
G.~Nakamura and H.~Wang.
\newblock The direct electromagnetic scattering problem from an imperfectly
  conducting cylinder at oblique incidence.
\newblock {\em Journal of Mathematical Analysis and Applications},
  397:142--155, 2013.

\bibitem{Ray}
N.~Raymond.
\newblock {\em Elements of spectral theory}.
\newblock IRMAR - Institut de Recherche Mathématique de Rennes, France, 2018.

\bibitem{RoitShef}
Y.~A. Roitberg and Z.~G. Sheftel.
\newblock Nonlocal boundary-value problems for elliptic equations and systems.
\newblock {\em Siberian Mathematical Journal}, 13(1):118--129, 1972.

\bibitem{ShaWu16}
Q.~C. Shang, Z.~S. Wu, Z.~J. Li, and L.~Bai.
\newblock Improvements for scattering from a large-sized chiral cylinder at an
  oblique incidence.
\newblock {\em J. Quant. Spectrosc. Ra.}, 162:50--55, 2015.

\bibitem{Sheftel}
Z.~G. Sheftel.
\newblock Energy inequalities and general boundary value problems for elliptic
  equations with discontinuous coefficients.
\newblock {\em Sib. Mat. Zh.}, 6(3):636--668, 1965.

\bibitem{Tsa18}
J.~L. Tsalamengas.
\newblock Oblique scattering from radially inhomogeneous dielectric cylinders:
  An exact volterra integral equation formulation.
\newblock {\em J. Quant. Spectrosc. Ra.}, 213:62--73, 2018.

\bibitem{WanNak12}
H.~Wang and G.~Nakamura.
\newblock The integral equation method for electromagnetic scattering problem
  at oblique incidence.
\newblock {\em Applied Numerical Mathematics}, 62(7):860--873, 2012.

\bibitem{Wloka}
J.~Wloka.
\newblock {\em Partial Differential equations}.
\newblock Cambridge University Press, Cambridge, 1987.

\bibitem{WloRowLaw}
J.~Wloka, B.~Rowley, and B.~Lawruk.
\newblock {\em Boundary value problems for elliptic systems}.
\newblock Cambridge University Press, Cambridge, 1st edition, 1995.

\bibitem{Yin17}
T.~Yin, G.~C. Hsiao, and L.~Xu.
\newblock Boundary integral equation methods for the two-dimensional
  fluid-solid interaction problem.
\newblock {\em SIAM J. Numer. Anal.}, 55(5):2361--2393, 2017.

\end{thebibliography}
\end{document}